\documentclass{book}
\usepackage[contrib,lang=british]{ems-book} 

\usepackage[all]{xy}
\usepackage{tikz}
\usepackage{tikz-cd}
\usepackage{float}
\usetikzlibrary{matrix,arrows,calc}
\usetikzlibrary{positioning}
\usetikzlibrary{decorations.pathreplacing,decorations.markings,decorations.pathmorphing}
\usetikzlibrary{positioning,arrows,patterns}
\usetikzlibrary{cd}
\usetikzlibrary{intersections}
\tikzset{
	every loop/.style={very thick},
	comp/.style={circle,black,draw,thin,inner sep=0pt,minimum size=5pt,font=\tiny},
	order bottom left/.style={pos=.05,left,font=\tiny},
	order top left/.style={pos=.9,left,font=\tiny},
	order bottom right/.style={pos=.05,right,font=\tiny},
	order top right/.style={pos=.9,right,font=\tiny},
	order middle right/.style={draw,thin,shape=rectangle,inner sep=2pt,outer sep=5pt,pos=.5,right,font=\tiny},
	order middle left/.style={draw,thin,shape=rectangle,inner sep=2pt,outer sep=5pt,pos=.5,left,font=\tiny},
	order node dis/.style={text width=.75cm},
	circled number/.style={circle, draw, inner sep=0pt, minimum size=12pt},
	below left with distance/.style={below left,text height=10pt},
	below right with distance/.style={below right,text height=10pt}
}
\def\ww{node[red]{$\mathbf{\scriptstyle\times}$}}
\def\nn{node{$\scriptstyle\bullet$}}

\def\be{\begin{equation}}   \def\ee{\end{equation}}     \def\bes{\begin{equation*}}    \def\ees{\end{equation*}}
\def\ba{\begin{align}} \def\ea{\end{align}}
\def\bas{\bes\begin{aligned}}  \def\eas{\end{aligned}\ees}
\def\={\;=\;}  \def\+{\,+\,} 
\newcommand\PP{\mathbb P}
\newcommand\bC{\mathbb C}
\newcommand\C{\mathbb C}
\newcommand\R{\mathbb R}
\newcommand\Z{\mathbb Z}
\newcommand\bZ{\mathbb Z}

\newcommand{\bH}{\mathbb{H}}
\newcommand{\halfplane}{\mathbb{H}}
\newcommand{\chalfplane}{\overline{\mathbb{H}}}
\newcommand{\bk}{k}%\mathbb{K}}
\newcommand\cl{\mathcal}
\newcommand{\calD}{\cl D}

\newcommand{\calA}{\cl A}

\newcommand\pvd{\operatorname{pvd}}

\newcommand{\modules}{\operatorname{mod}}

\newcommand{\Mod}{\operatorname{Mod}}
\newcommand{\Hom}{\operatorname{Hom}}
\newcommand{\Aut}{\operatorname{Aut}}
\newcommand{\Ext}{\operatorname{Ext}}

\newcommand\rep{\operatorname{rep}}
\newcommand\ext{\operatorname{ext}}

\newcommand{\sph}{\operatorname{sph}}

\newcommand{\GL}{\operatorname{GL}}
\newcommand{\rk}{\operatorname{rank}}
\newcommand{\torsion}{\mathcal{T}}
\newcommand{\torsionfree}{\mathcal{F}}

\newcommand{\Exch}{\operatorname{Exch}}

\newcommand\CY{CY}
\newcommand\stab{\operatorname{Stab}}
\newcommand\Stab{\operatorname{Stab}}

\newcommand{\Sim}{\operatorname{Sim}}

\newcommand{\Qgrad}{\bar{Q}}

\newcommand\bra{\langle}
\newcommand\ket{\rangle}
\newcommand{\del}{\partial}

\renewcommand{\Re}{\operatorname{Re}}
\renewcommand{\Im}{\operatorname{Im}}
\makeatletter \@addtoreset{equation}{section} \makeatother

\theoremstyle{plain}
\newtheorem{theorem}{Theorem}[section]
\newtheorem{lemma}[theorem]{Lemma}

\newtheorem{prop}[theorem]{Proposition}
\newtheorem{cor}[theorem]{Corollary}
\theoremstyle{definition}
\newtheorem{definition}[theorem]{Definition}

\newtheorem{rmk}[theorem]{Remark}

\newtheorem*{rmk*}{Remark}
\newtheorem*{pb*}{Problem}

\def\be{\begin{equation}}
\def\ee{\end{equation}}
\def\bes{\begin{equation*}}
\def\ees{\end{equation*}}
\def\ba{\be\begin{aligned}}
\def\ea{\end{aligned}\ee}
\def\bas{\bes\begin{aligned}}
\def\eas{\end{aligned}\ees}
\def\w{\mathbf{w}}
\newcommand{\Tri}{{\w\equiv\mathbf{1}}}
\newcommand\surf{\mathbf{S}}  % marked surface
\newcommand\surfo{{\mathbf{S}}_\Tri}  % decorated marked surface
\newcommand\sow{\surf_\w}  % decorated marked surface+w
\newcommand{\MM}{\mathbb{M}}

\newcommand{\EG}{\operatorname{EG}}
\DeclareDocumentCommand{\qmoduli}{ O{g} O{}}{{\mathrm{Quad}^{#2}_{#1}}}
\newcommand{\Quad}{\operatorname{Quad}}
\newcommand{\FQuad}{\operatorname{FQuad}}

\newcommand{\MCG}{MCG}

\DeclareMathAlphabet{\mathpzc}{OT1}{pzc}{m}{n}

\begin{document}

%\mainmatter

\title{Spaces of Bridgeland stability conditions in representation theory}
\titlemark{Spaces of Bridgeland stability conditions}

\emsauthor{1}{Anna Barbieri}{A.~Barbieri}

\emsaffil{1}{Universit\`a di Verona, Dipartimento di Informatica - Settore Matematica, 
Strada Le Grazie 15, 37134 Verona - Italy \email{anna.barbieri@univr.it}}

\begin{abstract}
The space of Bridgeland stability conditions is a complex manifold that can be attached to a triangulated category, of which it encodes some homological properties. These notes are an introduction to this topic, with a focus on examples from representation theory, and review the example of the Bridgeland--Smith correspondence for some quiver categories from marked surfaces.
\end{abstract}

\makecontribtitle

%\tableofcontents

\renewcommand{\thesection}{\arabic{section}} 

\section{Introduction}
The notion of stability in the
context of algebra and geometry is traditionally
interpreted as a classification tool to gather objects
(the stable or semi-stable ones) in
well-behaved moduli spaces.

Stability conditions for triangulated
categories were first introduced by Tom Bridgeland at the
beginning of 2000 in \cite{BrStab}.
One of the
major features of this notion is that by definition it
incorporates the possibility of considering the set of all stability
conditions as a complex manifold, denoted by $\stab(\calD)$,
attached to a triangulated category $\calD$, of which it encodes some
homological properties.
Since its introduction, the space
$\stab(\calD)$ has played a role in algebraic geometry,
representation theory, mirror symmetry, and some branches of
mathematical physics, providing interesting synergies. While
these spaces are unknown in many cases, there are examples that are
quite well understood.

The goal of these notes
is to give an introduction to
spaces of stability conditions on triangulated
categories --with a view towards module categories.
As an example, we consider the space of stability conditions of
a class of three-Calabi--Yau
categories from quivers with potential, which are
well known
in representation and cluster theory.
Throughout the chapter, instead of giving entire proofs, we
try to emphasise the main ideas and ingredients or give
references.

The chapter is organised as follows. In Section \ref{sec_stab_def}, we recall the definition of Bridgeland
stability conditions as it is currently predominantly accepted. The space $\Stab(\calD)$ is introduced in Section \ref{sec_stab_mfd}, together with its main properties
as a topological and complex
manifold. We recall how the geometry of the space is controlled by bounded t-structures on the category, and we briefly mention some
research directions
that have received attention in the last decade,
concerned with the stability manifold itself.
Section \ref{sec_ginzburgexample} is aimed at reviewing the
computation, due to Bridgeland and Smith, of the space of stability conditions on
some quiver categories from marked surfaces, summarised in Theorem \ref{thm:BS15_iso}
as an isomorphism involving
moduli spaces of framed quadratic
differentials on a weighted marked surface. This example is the most familiar to the author, and it is used to see in practice some of the ingredients from Section \ref{sec_stab_mfd} and to give a hint on possible
fruitful interactions between Bridgeland stability conditions and other moduli problems.
The relevant categories and the necessary notions
of quadratic differentials from the theory of flat surfaces
are briefly recalled, for consistency.

The material presented in this article reflects the
research interests of the author, and there are
therefore many interesting aspects and
directions that are not covered or mentioned. These include, for instance, the problem of
constructing
stability conditions for derived categories of
varieties (for which surveys are nevertheless available), and applications to algebraic geometry;
enumerative theories and questions about defining moduli spaces of objects associated
with Bridgeland stability conditions; wall-crossing
phenomena in broad sense.

%%%%%%%%%%%%%%%%%%%%%%%%%%%%%%%%%%

\section{Stability conditions}\label{sec_stab_def}

\subsection{Preliminary definitions}
We fix here some notation that will be used throughout the section
and the whole paper. $\bk$ is an algebraically closed field, usually $k=\C$, and
any category is additive, $\bk$-linear, and essentially small.
A short exact sequence (s.e.s.) $0\to A \to B \to C \to 0$ in an
abelian category is often represented as $A \to B \to C$,
while a distinguished triangle in a triangulated category is represented either as $A \to B \to C \to A[1]$ or as an usual triangle. If two objects are isomorphic, we say they belong to the same iso-class.
Given subcategories $\cl H_1,\cl H_2$ of an abelian or a
triangulated category~$\cl C$, and a set of objects $\cl B$, we define the following subcategories:
\begin{align*}
    \cl H_1*_{\cl C}\cl H_2&:=\big\{M\in\cl C\mid\text{$\exists$ s.e.s. (or triangle) }T\to M\to F(\to T[1]) \\
    &\phantom{:=M\in\cl C\mid\text{$\exists$ s.e.s. (or triangle)}T\to M \to}\text{s.t. $T\in\cl H_1, F\in\cl H_2$} \big\}\,,\\
    \cl B^{\perp_{\cl C}}&:=\big\{C\in\cl C:\Hom_{\cl C}(B,C)=0,\ \forall\,B\in \cl B\big\}\, \text{ and similarly }\,{^{\perp_\cl C}}\cl B\,,\\
    \cl H_1\perp_{\cl C}\cl H_2&:=\cl H_1*\cl H_2\,\, \text{ if }\,\cl H_2=\cl H_1^{\perp_{\cl C}}\, \text{ and }\, \cl H_1={^{\perp_{\cl C}}\cl H_2}\,.
\end{align*}

We denote by $\bra \cl B\ket_{\cl C}$ the closure under extensions and possibly shifts of $\operatorname{Add}\cl B$ in $\cl C$, and we  say that it is the subcategory \emph{generated} by $\cl B$.
We will usually omit the subscript~$\cl C$.

\begin{definition}\label{def_finiteheart}An abelian category is called of \emph{finite length} if, for any $E\in \calA$, there is a finite sequence $0=E_0\subset E_1\subset\dots\subset E_n =E$ such that all $E_i/E_{i-1}$ are simple. \\	
    It will be called \emph{finite} if, moreover, it has a finite number of iso-classes simple objects.
\end{definition}

We recall that the Grothendieck group of an abelian (resp., triangulated) category $ \cl C$ is
the group generated by the classes $[-]$ of isomorphism of objects in $\cl C$, modulo relations
induced by short exact sequences (resp., distinguished triangles):
\[A \to B \to C (\to A[1])\quad \text{implies}\quad [B]=[A]+[C].\]
It is denoted by $K(\cl C)$. It is easy to verify that $\left[A[-1]\right]= -\left[A\right]$ when $\cl C$ is triangulated. If
$\cl C$ is an abelian category and $A\in \cl C$, the class $-[A]$ is not represented
by any object in $\cl C$.

If a triangulated category $\calD$ is $\Hom$-finite, that is, for any $E,F\in\calD$, the vector space $\oplus_i\Hom_\calD(E,F[i])$ is finite dimensional, the Euler form $\chi:K(\calD)\times K(\calD) \to \bZ$ is defined by
\[\chi([E],[F])=\sum_i(-1)^i\dim \Hom_\calD(E,F[i]).
\]

The notion of a t-structure for a triangulated category was introduced in \cite{f-pervers} by A.~Beilinson, J.~Bernstein, P.~Deligne (t-category), and refined to the notion of a slicing in \cite{BrStab} by T.\ Bridgeland. We are interested here in \emph{bounded} t-structures, which are non-degenerate t-structures modelled on the decomposition of the bounded derived category $\calD^b(\calA)$ of an abelian category $\calA$ into objects with only non-positive non-zero cohomology $H^i(E)=0$, $i>0$, only non-negative non-zero cohomology $H^i(E)\neq 0$, $i<0$, and their extensions.

\begin{definition}
    A \emph{bounded t-structure} on a triangulated category $\calD$
    is defined by a full subcategory
    $\mathcal{L}\subset \calD$ (called the aisle), closed under shift  $\mathcal{L}[1]\subset\mathcal{L}$,
    such that
    \begin{align}
        \label{tstr1}
        \calD&=\cl L \perp \cl L^\perp,\quad \text{and moreover}\\
        \label{tstr2}
        \calD&=\bigcup_{i,j\in\Z}\mathcal{L}[i]\cap\mathcal{L}^{\perp}[j].
    \end{align}
    The \emph{heart} of a bounded t-structure $\mathcal{L}\subset\calD$ is the full subcategory $\calA=\mathcal{L}\cap\mathcal{L}^{\perp}[1]\subset \calD$.
\end{definition}
\begin{lemma}[{\cite[\S 1.3]{f-pervers}}] The heart of a bounded t-structure is an abelian category and it determines the bounded t-structure as the extension-closed subcategory generated by the
    subcategories $\calA[j]$ for integers $j\geq 0$.
\end{lemma}

In the rest of the text, we will therefore use interchangeably the notion of a bounded
t-structure or its heart. While it is clear that if $\cl L$ is a bounded t-structure, then $\cl L[n]$ is also a bounded t-structure for any integer $n$, we easily find t-structures that are not the shift of one another.  A typical example is provided by the bounded derived category of the representation of the $(n+1)$-th Beilinson's quiver $B_{n+1}$
\[\xymatrix{
    B_{n+1}\ar@{}[r]|=& \bullet_1 \ar@/^1pc/[rr]\ar@/_.8pc/[rr] & \vdots{\scriptscriptstyle{n+1}} &\bullet_1  \ar@{}[r]|{.\;.\;.\,}  &\bullet_{n} \ar@/^1pc/[rr]\ar@/_.8pc/[rr] & \vdots{\scriptscriptstyle{n+1}} &\bullet_{n+1}
}\]
which has $n+1$ vertices and $n+1$ arrows between any two consecutive vertices:
\[\rep(B_{n+1})\subset \calD^b(\rep(B_{n+1}))\simeq \calD^b(\C\PP^n)\supset \operatorname{Coh}\C\PP^n.\]
%see \cite{alcrew}.
Indeed, any abelian category $\calA$ is the heart of
a bounded t-structure in its bounded derived category $\calD^b(\calA)$. In the example, $\rep(B_{n+1})$ is a finite heart, while the abelian category $\operatorname{Coh}(\C\PP^n)$ of coherent sheaves on the complex projective space $\C\PP^n$ is not.

A way of producing new t-structures, which are not necessarily standard in the sense above, is via tilting at a torsion pair.
\begin{definition}A \emph{torsion pair} in an abelian category $\cl H$ is a pair of subcategories $(\torsion, \torsionfree)$ such that $\cl H=\torsion\perp\torsionfree$. We call $\torsion$ the torsion class and $\torsionfree$ the torsion-free class.
\end{definition}
A torsion pair in the heart of a
bounded t-structure $\cl H$ in a $\calD$ defines new bounded t-structures with hearts
\[ \mu_{\torsionfree}^\sharp\cl H:= \torsion \perp_\calD \torsionfree[1], \quad \quad
\mu_{\torsion}^\flat\cl H:=\torsionfree\perp_\calD \torsion[-1].
\]
They are called the \emph{forward tilt} at $\torsionfree$
and the \emph{backward tilt} at $\torsion$, respectively, and are related by $\mu^\sharp_{\torsion[-1]} \mu^\flat_\torsion\cl H=\cl H$ and $\mu^\flat_{\torsionfree[-1]} \mu^\sharp_\torsionfree\cl H=\cl H$, \cite{tiltingbook}. When we tilt at a torsion(\mbox{-}free) class $\bra S\ket$
generated by a simple object $S\in \cl H$, we speak about a simple tilt
and we simplify the notation to
\[\mu^\sharp_S\cl H \quad \text{and} \quad \mu^\flat_S\cl H.\]
\begin{definition}
    The \emph{exchange graph} $\EG(\calD)$ of a triangulated category $\calD$ is the graph whose vertices are finite hearts of bounded t-structures on $\calD$ and whose arrows are either forward or backward simple tilts.
\end{definition}
For the purposes of these notes, we will usually consider forward
tilts for $\EG(\calD)$, though clearly this only affects the direction of the
arrows.

The group $\Aut(\calD)$ acts on $\EG(\calD)$. Since any autoequivalence commutes with the shift functor,
if $\Phi\in\Aut(\calD)$ and $\calA$ is the heart of a bounded t-structure, then
\begin{equation}\label{phi_mu}\Phi\left(\mu^{\flat/\sharp}_{\torsion/\torsionfree}(\calA)\right)
    \= \mu^{\flat/\sharp}_{\Phi(\torsion/\torsionfree)} \Phi(\calA).
\end{equation}

The following lemma characterises bounded t-structures of a triangulated category. The proof can be deduced from \cite[\S 1.3]{f-pervers} and is sketched below.

\begin{lemma}[{\cite[Lemma 3.2]{BrStab}}]\label{decomp_heart}
    Let $\calA\subset\calD$ be a full additive subcategory of a triangulated category $\calD$.
    Then $\calA$ is the heart of a bounded t-structure $\mathcal{L}\subset\calD$ if and
    only if the following two conditions hold:
    \begin{itemize}
        \item[(a)]if $k_1>k_2$ are integers, and $A$ and $B$ are objects of
        $\calA$, then
        \[\Hom_{\calD}(A[k_1],B[k_2])=0;\]
        \item[(b)]
        for every nonzero object $E\in\calD$, there is a finite
        sequence of integers
        \[k_1>k_2>\cdots>k_n\]
        and a collection of triangles
        \[\label{cohomology}
        \xymatrix@C=.5em{
            0_{\ } \ar@{=}[r] & E_0 \ar[rrrr] &&&& E_1 \ar[rrrr] \ar[dll] &&&& E_2
            \ar[rr] \ar[dll] && \ldots \ar[rr] && E_{n-1}
            \ar[rrrr] &&&& E_n \ar[dll] \ar@{=}[r] &  E_{\ } \\
            &&& A_1 \ar@{-->}[ull] &&&& A_2 \ar@{-->}[ull] &&&&&&&& A_n \ar@{-->}[ull]
        }
        \]
        with $A_j\in\calA[k_j]$ for all $j$.
    \end{itemize}
\end{lemma}
The objects $A_j$ appearing in \eqref{cohomology} are called the $k_j$-th \emph{cohomology} class of $E$ \emph{with respect to the bounded t-structure}. They are unique up to isomorphism \cite{BrStab}.
\begin{proof}
    For one direction, we consider $\cl L=\bra \calA[i],\, i\geq 0\ket_\calD\ket $ and
    $\cl G=\bra \calA[-i],\, i\geq 1\ket_\calD\ket$. By conditions (a) and (b), $\cl G = \cl L^\perp$.
    Let $E\in \calD$ and
    $m$ be the greatest integer among the $i$ in $\{1,\dots, n\}$ such that $k_i\geq 0$.
    Then the cone of the non-zero
    composite functor $E_m\to E$ lies in $\bra \calA[j], k_{m+1}\geq j\geq k_n \ket\subset\cl G$, and we have a decomposition $E_m \to E \to G\to E_m[1]$ with $G\in\cl L^\perp$.

    The other implication can be proved by using the truncation functors  $\tau_{\geq 0}, \tau_{\leq 0}$ and their shifts $\tau_{\geq k}$, $\tau_{\leq k}$, which are defined in \cite[\S 1.3, see in particular Propositions~\mbox{1.3.3--1.3.5}]{f-pervers}. Theorem~1.3.6 in op.\ cit.\ shows moreover that $H^k:=\tau_{\geq k} \tau_{\leq k}:\calD\to \calA$ is a cohomological functor. It associates $E\in \calD$ with the shifted subfactor $A[-k]\in\calA$ appearing in (b).

    Last, condition \eqref{tstr2} is equivalent to finiteness of the sequence of triangles appearing in (b).
\end{proof}

\begin{cor}\label{isogroth} If $\calA$ is the heart of a bounded t-structure on a triangulated category $\calD$, then there is an isomorphism of Grothendieck groups
    \[K(\calA)\simeq K(\calD).\]
\end{cor}
\begin{proof} Short exact sequences in $\calA$ are precisely the distinguished triangles in
    $\calD$ with three vertices in $\calA$. The map $K(\calA)\to K(\calD)$ is induced by the inclusion $\calA \subset \calD$, while its inverse sends $[E]_\calD$ to the alternate (finite) sum $\sum_{i\in \bZ}(-1)^{k_i}\big[A_i[-k_i]\big]_\calA$, for $A_i$ and $k_i$ defined in Lemma \ref{decomp_heart}, (b).
\end{proof}

\begin{definition}[{\cite[Definition 3.3]{BrStab}}]\label{def_slicing}A \emph{slicing}
    on the triangulated category $\calD$ is a family
    of full additive subcategories $\cl P:=\{\cl P(\phi)\}_{\phi\in\R}
    \subset \calD$ such that
    \begin{enumerate}
        \item[(a)] $\cl P(\phi+1)=\cl P(\phi)[1]$ for all $\phi\in\R$;
        \item[(b)] if $\phi_1>\phi_2$ and $A_j\in\cl P(\phi_j)$, $j=1,2$, then $\Hom_\calD(A_1,A_2)=0$;
        \item[(c)] for any non-zero object $E\in\calD$, there is a finite sequence of real numbers $\phi_1>\phi_2>\dots>\phi_m$ and a collection of distinguished triangles
        \[\xymatrix@C=1pc@R=1pc{ 0\ar@{}[r]|{=}&E_0 \ar[rr] && E_1 \ar[rr]\ar[dl] && E_2 \ar[r]\ar[dl] & \dots \ar[r] & E_{m-1} \ar[rr] && E_m \ar[dl] \ar@{}[r]|{=} & E\\
            & & A_1 \ar@{-->}[ul] && A_2 \ar@{-->}[ul] &&&& A_m\ar@{-->}[ul]
        }\]
        with $A_j\in\cl P(\phi_j)$ for all $j=1,\dots,m$.
    \end{enumerate}
\end{definition}
It is not required by definition that $\cl P(\phi)\neq \{0\}$ for all $\phi \in \R$, nor for the non-trivial slices to be dense in $\R$.

As in Lemma \ref{decomp_heart}, the decomposition of axiom (c) is unique up to isomorphism; hence
one can define $\phi_{\cl P}^+(E)=\phi_1$ and $\phi_{\cl P}^-(E)=\phi_n$. For any interval $I\subset \R$, Bridgeland defines
\[\cl P(I):=\bra \cl P(\phi)\mid \phi\in I\ket_\calD.\]
It coincides with the subcategory $\bra E\in\calD \mid \phi_{\cl P}^{\pm}(E)\in I\ket_\calD$.
\begin{lemma}Suppose $I=(0,1]$ and $\lambda\in I$. Then
    \begin{enumerate}
        \item $\cl P(I)$ is the heart of a bounded t-structure on $\calD$, and
        \item $\cl P((\lambda, 1])\perp\cl P((0,\lambda])$ is a torsion pair in $\cl P(I)$.
    \end{enumerate}
\end{lemma}
\begin{proof}Both statements follow from the definitions and the conditions of Lemma \ref{decomp_heart}, using a truncation functor.
\end{proof}
The result actually holds for any interval of length $1$. In particular $\cl P(I)$ is abelian if $I$ has length $1$. It is quasi-abelian if $I$ has length less than $1$, \cite{BrStab}.
We say that $\cl P((0,1])$ is the \emph{heart} of the slicing $\cl P$.

\subsection{Bridgeland stability conditions}

The notion of a stability condition on a triangulated category was
introduced  in \cite{BrStab}. The definitions given below (\ref{def_stab1} and \ref{def_stab2})
are the mostly used
currently, see also the series of papers by Bayer, Macr{\`\i, Stellari, and co-authors}. The equivalence of the two definitions is sketched in Theorem \ref{thm_equiv}, \cite[Proposition 5.3]{BrStab}. The differences with the original definition involve the support condition
and the possible dependence on a finite rank lattice.

We start with the preliminary definition of a stability function on an abelian category and of the Harder--Narasimhan condition.

\begin{definition}\label{stabfunct} Let $\calA$ be an abelian category and $Z\in \Hom(K(\calA),\C)$
    such that, for any $0\neq A\in\calA$,
    \[Z([A])\in \chalfplane:=\{re^{\pi i \theta}\in\R \mid r\in \R_{>0},\ 0<\theta\leq 1\}.\]
    We say $\frac{1}{\pi}\arg Z([A])$ is the \emph{phase} of $A$.
    \begin{enumerate}
        \item An object $A\in\calA$ is said to be \emph{$Z$-semistable} if, for any non-zero proper sub-object $B\hookrightarrow A$, we have $\frac{1}{\pi}\arg Z([B])\leq \frac{1}{\pi}\arg Z([A])$. It is called $Z$-\emph{stable} if the inequality holds strictly.
        \item $Z$ is said to be a \emph{stability function} if it satisfies the \emph{Harder--Narasimhan property}:
        for any $0\neq A\in\calA$, there is a finite chain of sub-objects
        \[0\simeq A_0\subset A_1\subset \dots \subset A_m=A\]
        whose quotients $F_j=A_j/A_{j-1}$ are $Z$-semistable of strictly decreasing phases.
    \end{enumerate}
\end{definition}

\smallskip\par
Let $\calD$ be a triangulated category. We fix a finite rank free lattice
$( \Lambda, \bra-,-\ket )$, i.e., a free abelian group $\Lambda$ equipped with an inner product
$\bra -,-\ket$, together with a surjective group homomorphism
$\nu:K(\calD)\twoheadrightarrow \Lambda$. If $\calD$ is $\Hom-$finite and $K(\calD)\simeq \bZ^{\oplus n}$, we take $\big(K(\calD), \chi(-,-)\big)\stackrel{id}{=}\big(\Lambda, \bra-,-\ket\big)$. In many cases, if $K(\calD)$ has not finite rank,  it is standard to consider central charges that factor through the numerical part, i.e., the quotient of $K(\calD)$ by the null space of the Euler form on $\calD$, or, for some $\calD=\calD^b(\operatorname{Coh}(X))$, the singular cohomology $H^*(X,\Z)$. In the next definition we use the isomorphism of Grothendieck groups of Corollary \ref{isogroth}.

\begin{definition}\label{def_stab1}
    A \emph{stability condition} $\sigma$ on $\calD$, \emph{supported} on the heart $\calA$, is a pair
    \[\sigma=(\calA,Z),\]
    consisting of the heart of a bounded t-structure $\calA$ on $\calD$,
    together with a \emph{stability function} $Z$ on $\calA$ that factors through $\Lambda$
    \[Z:K(\calA)\stackrel{\nu}{\twoheadrightarrow}\Lambda \rightarrow \C,\]
    satisfying
    the \emph{support property}: %\\
    there exists a norm $\|\cdot\|$ on $\Lambda\otimes \R$ and a constant $c\in\R_{>0}$ such that, for any $Z$-semistable $0\neq A\in \calA$, $|Z(A)|\geq c\|\nu[A]\|$.
    \par
    The homomorphism $Z$ is referred to as the \emph{central charge}.
\end{definition}

Note that, since $\Lambda\otimes \R$ is a finite dimensional vector space, all norms are equivalent; hence any definition depending on definition \ref{def_stab1} will not depend on the choice of the norm.

While the support property looks at a first glance somehow arbitrary and with a different
flavour compared with the rest of the definition, it is
crucial in order to define a topology on the set of all stability conditions.
It was introduced by M.~Kontsevich and Y.~Soibelman in \cite{KS},
where the authors also show that it can be equivalently expressed as the the condition for which there exists a quadratic form $Q:\Lambda\otimes_\Z\R \to \R$ such that
\begin{itemize}
    \item the kernel of $Z$ is negative definite with respect to $Q$, and
    \item  $Q(\nu[A])\geq 0$ for any $Z$-semistable object $A$.
\end{itemize}
Indeed, if, for $\alpha\in\Lambda\otimes_\Z\R$, one
writes $\|\alpha\|=\alpha\cdot \alpha$, then we can define
$Q(\alpha,\beta)=\sqrt{Z(\alpha)\overline{Z}(\beta)}- \alpha\cdot \beta$. On the other hand, given
$Q(\alpha,\alpha)$, one easily sees that $\|\alpha\|=|Z(\alpha)|-Q(\alpha)$ is a norm on
$K(\calA)$. This is most useful when it boils down to a Bogomolov--Gieseker type inequality (see, e.g., \cite{BMS}).
\smallskip\par
A Bridgeland stability condition on a triangulated category is also defined in terms of a slicing parametrising \lq\lq distinguished\rq\rq\ objects: the semistable ones. Note that here, as well as in Definition \ref{def_stab1}, we drop from the notation
the dependence of $\sigma$ on the choice of $(\Lambda,\nu)$.

\begin{definition}\label{def_stab2} Let $K(\calD)\stackrel{\nu}{\twoheadrightarrow}\Lambda$ as above.
    A \emph{stability condition} on $\calD$ is a pair
    \[\sigma \=(\cl P,Z),\]
    where $\cl P$ is a slicing on $\calD$ and $Z\in\Hom\left(K(\calD),\C\right)$
    is a group homomorphism that factors through $K(\calD)\stackrel{\nu}{\rightarrow}\Lambda$ and satisfies the support property and the following
    \emph{compatibility} condition: if $0\neq E\in\cl P(\phi)$, then there exists $m(E)\in\R_{>0}$ such
    that $Z([E])=m(E)\exp(i\pi\phi)$.
\end{definition}
The following definition is well-posed thanks to Lemma \ref{abelianPphi}.
\begin{definition} The non-zero objects $0\neq E \in\cl P(\phi)$ are said to be \emph{$\sigma$-semistable of phase $\phi$}, and the simples in $\cl P(\phi)$ are said to be \emph{$\sigma$-stable}.
\end{definition}

\begin{lemma}[{\cite[Lemma 5.2]{BrStab}}]\label{abelianPphi} If a slicing is compatible with a central charge, then any $\cl P(\phi)$ is an abelian category of finite length.
\end{lemma}
\begin{proof}One can prove that $\cl P(\phi)$ is abelian by showing that it is closed under
    kernels and cokernels inside the abelian category $\cl P((\phi-1,\phi])$. First show by contradiction that if
    $E \to F \to G \to E[1]$ is a distinguished triangle in $\cl P((\phi-1,\phi])$, then
    $\phi^+(E)\leq \phi^+(F)$ and $\phi^-(F)\leq \phi^-(G)$. Then use the compatibility condition.
    The finite length property is ensured by the support property of the central charge.
\end{proof}

\begin{theorem}\label{thm_equiv}Definition \ref{def_stab1} and Definition \ref{def_stab2} are equivalent. The $\sigma$-(semi)stable objects of $\calD$ are exactly the $Z$-(semi)stable objects of $\cl P((0,1])$ and all their shifts.
\end{theorem}
\begin{proof}[Sketch of the proof]
    If we have a stability condition $\sigma = (\cl A,Z)$ in the sense of Definition \ref{def_stab1}, then the collection $\cl P := \left\{\cl P(\phi), \phi\in\R\right\}$ defined by
    \[\cl P(\phi) := \{E  \in \calA,\ Z\text{-semistable in} \; \calA \; \text{of phase} \; \phi \}\cup \{\text{zeroes}\}, \quad 0<\phi\leq 1,\]
    and
    \[\cl P(\phi+1)=\cl P(\phi)[1], \]
    is a slicing on $\calD$, compatible with $Z$ regarded as a group homomorphism on $K(\calD)$.

    On the other hand,
    a slicing $\cl P$ defines a bounded t-structure $\calD_{> 0}:=\cl P\left((0,+\infty)\right)$ on $\calD$ with heart $\cl P\left( (0,1]\right)$, and $Z\in\Hom(K(\calD),\C)$ induces a stability function on $\cl P\left( (0,1]\right)$.
\end{proof}
From now on, with slight abuse of notation, we will write $Z(E)$ for $Z([E])$.
For any non-zero object $E$ in $\calD$, one
defines its \emph{mass} with respect to a chosen stability condition $\sigma=(\cl P, Z)$ as
\[m_{\sigma}(E)=\sum_j|Z(A_j)|\in\R_{>0}\,,\]
where the $A_j$ are the Harder--Narasimhan factors with respect to $\sigma$, i.e., the objects, unique up to isomorphism, appearing in Definition \ref{def_slicing}, c) for the underlying slicing. Note that the mass of an object cannot vanish, but the central charge can vanish.
If a non-zero object $X\in\calD$ is $\sigma$-semistable and belongs to $\cl P(\phi)$, then $\phi^+(X)=\phi^-(X) = \phi$, and $Z(X)=m_{\sigma}(X)\exp(\pi i \phi)$. Choosing $\arg z \in (0,2\pi]$ for any $z\in\C^*$ as a standard branch for the
logarithmic function, we always have that if $0\neq X\in \calD$ is $\sigma$-semistable of phase $\phi$, then $\phi-\frac{1}{\pi}\arg Z(X)\in\Z$ with $\phi=\frac{1}{\pi}\arg Z(X)$ if $X\in\cl P((0,1])$. This is true in particular for all the simple objects of the supporting heart.
\smallskip\par
The set of all Bridgeland stability conditions on a triangulated category for a fixed choice of $(\Lambda, \nu)$ is denoted by $\stab_{(\Lambda, \nu)}(\calD)$.
Even when a stability condition on a given triangulated category is known to exist, computing the whole $\stab_{(\Lambda, \nu)}(\calD)$ can be very hard.

\subsection{Stability functions}\label{sec:stab:examples}
The notion of Bridgeland stability conditions on a triangulated category was inspired by the
work of Douglas \cite{douglas1, douglas2} on $\Pi$-stability for D-branes, and, more in general,
by ideas from string theory. These ideas have driven part
of the mathematical research on the stability manifold since its definition.

On the other hand, Bridgeland stability provides the first example of stability conditions on a
triangulated category, and choosing a central charge on a heart appears like a natural generalisation of previously known notions of stability conditions
on an abelian category, their key property being the Harder--Narasimhan property. The typical example is slope stability, but it is not always true that stability in abelian sense can be promoted to Bridgeland stability.

\subsubsection*{Slope stability}

We consider slope stability defined by Alastair King for the abelian category of representations of quivers and module categories. Let us take $Q$ an acyclic finite quiver, and $\underline{V}=(V_i,f_\alpha)_{i,\alpha}$ a representation of $Q$. Fix $\underline{a}\in\bZ^{|Q^0|}$ such that $\sum_i a_i d_i=0$ for some dimension vector $\underline{d}$, and set
\[\mu_{\underline{a}}(\underline{V})=\frac{\sum_i a_i \dim V_i}{\sum_i \dim V_i}.\]
We say that a representation $\underline{V}$ is $\mu_{\underline{a}}$-semistable if $\mu_{\underline{a}}(\underline{V})= 0$ and, for any sub-represen\-tation $\underline{W}\subseteq\underline{V}$, $\mu_{\underline{a}}(\underline{W})\geq 0$. It is called $\mu_{\underline{a}}$-stable if the only sub-representations with $\mu_{\underline{a}}(\underline{W})=0$ are the trivial ones. The key result by King \cite{king} concerns the
existence of moduli spaces of $\mu_{\underline{a}}$-semistable
$Q$-representations of fixed dimension $\underline{d}$ as a
projective variety. It is done using GIT techniques.

In general, a slope function on $\rep(Q)$ is given by two additive functions $c:\rep(Q)\to\R$ and $r:\rep(Q)\to\R_{>0}$ as $\mu(\underline{V})=\frac{c(\underline{V})}{r(\underline{V})}$, and mimic the analogous notion by Mumford for vector bundles (where $c$ and $r$ are the degree, depending on the choice of a polarisation on $X$, and the rank, respectively), extended to the abelian category of coherent sheaves  over a curve $X$.

The slope function satisfies the Harder--Narasimhan property, i.e., for any $\underline{V}\in\rep Q$, there exist $\underline{F}^k=\underline{V}\supset  \underline{F}\supset \dots \supset \underline{F}^0=0$ such that
\[\underline{F}^j/\underline{F}^{j-1}\]
are semistable of decreasing slope.
Positivity and finiteness properties allow us to regard a slope function on $\rep(Q)$ as a stability function in the
sense of Definition \ref{stabfunct} by setting $Z_\mu(\underline{V})=-c(\underline{V})+i r(\underline{V})$. Similarly for coherent sheaves of pure dimension on a curve $X$, taking $\Lambda=H^0(X,\Z)\oplus H^2(X,\Z)$. Note, however, that, if
$X$ is not a curve, this argument doesn't work.

A systematic study of stability functions and the Harder--Narasimhan property in the abelian context was carried out by Rudakov. See, for example, \cite{rudakov} and subsequent papers.

\subsubsection*{Finite hearts}
A special case is that of finite abelian categories. Suppose $\calD$ has
a bounded t-structure with a finite heart $\cl H$.
Let $\Sim(\cl H)=\{[S_1],\dots, [S_n]\}$ be a maximal set of iso-classes of simple objects of $\cl H$. Then $K(\calD)\simeq \Z^{\oplus n}$, and any
group homomorphism $Z\in\Hom(K(\cl H),\C)$ such that $Z(S_i)\in\halfplane$,
automatically satisfies the Harder--Narasimhan
condition and the support property, and therefore is a stability function on
the heart $\cl H$. The non-trivial property is
the Harder--Narasimhan condition. See \cite[{Section 1}]{rudakov} or \cite[{Proposition 2.4}]{BrStab} for the proof under the weaker assumption that there are no infinite chains of subobjects (resp., quotients) with increasing (resp., decreasing) value of $\phi=\frac{1}{\pi}\arg Z$, whose key points are summarised below. The following hold:
\begin{itemize}
    \item[(i)] any simple object is $Z$-semistable, and since any descending
        chain of subobject  and any ascending chain of quotients
        stabilises, then, for any $0\neq E\in \cal H$, there exist a
        $Z$-semistable subobject $0\neq A\subset E$ with $\phi(A)\geq \phi(E)$
        and a $Z$-semistable quotient $E\twoheadrightarrow B\neq 0$ with
        $\phi(E)\geq \phi(B)$;
    \item[(ii)] (see-saw property) if $A\to E \to B$ is a short exact sequence,
        then $\phi(A)\geq \phi(E)$ if and only if $\phi(E)\geq \phi(B)$;
    \item[(iii)]  there is no non-zero map $A\to B'$ if $A,B'$ are semistable
        with $\phi(A)>\phi(B')$.
\end{itemize}
In the finite abelian category $\cl H$, any semistable object has trivial Harder--Narasimhan filtration, and we
can work inductively on the length of an object.
The strategy is to show that, for any $E\in\cl H$, there
exists a maximally destabilizing quotient, that is, $E\twoheadrightarrow B_E\neq 0$ with $\phi(E)\geq\phi(B_E)$
such that, for any semistable quotient $E\twoheadrightarrow B'\neq 0$, we have $\phi(B')\geq \phi(B_E)$, and the equality
implies that $E\twoheadrightarrow B'$ factors through
$E\twoheadrightarrow B'$. If such an object exists, it is semistable thanks to the second part of (i), and will play the role of the minimal phase quotient in the Harder--Narasimhan filtration.

If $E$ is not
$Z$-semistable, we take an arbitrary $Z$-semistable subobject $A$ with
$\phi(A)>\phi(E)$, and a short exact sequence $A\to E \to E'$.
Assuming the existence of a $Z$-semistable maximally destabilizing
object $B_{E'}$ of $E'$, and using (i)--(iii), we can show that
$B_{E'}$ is also a maximally destabilizing object for $E$.
Then the
nine
lemma and the see-saw property imply that
if $A'=\ker(E\to B_E)$ and $B_{A'}$ is a maximally destabiling
quotient of $A'$, then $\phi(B_{A'})>\phi(B_E)$. This allows to
construct a Harder--Narasimhan filtration.

%%%%%%%%%%%%%%%%%%%%%%%%%%%%%%%%%%

\section{The stability manifold}\label{sec_stab_mfd}
Let $\calD$ be a $\Hom$-finite triangulated category, $\nu:K(\calD)\twoheadrightarrow \Lambda$ as in
Section~\ref{sec_stab_def}. The set of Bridgeland stability conditions
on $\calD$ factoring through $K(\calD)\stackrel{\nu}{\to}\Lambda$ here is denoted by
\[\stab(\calD)=\stab_{(\Lambda,\nu)}(\calD).\]
We remove the dependence on $(\Lambda,\nu)$ also from the notation for a single stability condition.

The main result in \cite{BrStab} is that $\stab(\calD)$ can
be given the structure of a complex manifold.
The goal of this section is to review the complex structure
and the main properties of the stability manifold, to provide
a few well-known examples, and to introduce some old and new
questions. For simplicity, and abusing
notation, we use the
expression \lq\lq central charge\rq\rq both for the map $Z:K(\calD) \to \C$ and for the induced map $\Lambda \to \C$.

\subsection{The complex structure}

The map $d: \stab(\calD)\times \stab(\calD)\to \R_{\geq 0}\cup \{+\infty\}$ defined by
\be\label{eq_metric} d(\sigma_1,\sigma_2)=
\sup_{0\neq E\in\calD}\left\{
|\phi_{\sigma_2}^-(E)-\phi_{\sigma_1}^-(E)|\,,\,
| \phi_{\sigma_2}^+(E)-\phi_{\sigma_1}^+(E)|\,,\,
\Big|\log\frac{m_{\sigma_2}(E)}{m_{\sigma_1}(E)}\Big|\right\}
\ee
is a generalised metric on $\stab(\calD)$, i.e., it satisfies the axiom of a metric space
except that it need not be finite \cite{BrStab}. We will loosely refer to it as a metric.  As a consequence, it defines a topology on $\stab(\calD)$
and induces a metric space structure on each connected component. We consider
$\stab(\calD)$ as endowed with the metric topology.
Equivalently, the topology is induced by the generalised metric
\[ d(\sigma_1,\sigma_2)= \sup_{0\neq E \in \calD} \left\{
|\phi_{\sigma_2}^-(E)-\phi_{\sigma_1}^-(E)|\,,\,
| \phi_{\sigma_2}^+(E)-\phi_{\sigma_1}^+(E)|\,,\,
\|Z_1-Z_2\|_{\Lambda_\C^*}
\right\},\]
where $\|W\|_{\Lambda_\C^*}$ denotes the operator norm on $\Hom_\Z(\Lambda,\C)$. It is easy to relate the use of the operator norm here with the support property of Definition \ref{def_stab1}, which can be rewritten as
\[ \inf\left\{\frac{|Z(E)|}{\|\nu[E]\|_{\Lambda_\R}}: 0\neq E \text{ semistable}\right\}>0.
\]
According to the metric $d$ defined above, the distance between two stability conditions
depends both on how \lq\lq different\rq\rq\ the central charges are, and how further
apart the hearts of the slicings are. For instance, if two stability conditions $\sigma_1=(\calA,Z)$ and $\sigma_2=(\calA[2n],Z)$ differ
by the choice of shifted hearts of bounded t-structures, then their distance is $2n$. On each connected component the generalised metric defined in \eqref{eq_metric} is finite
and complete  (see \cite{BMS,woolf2} for details).
Some metric properties of the stability manifold have been
studied by Woolf, \cite{woolf2}.

\begin{theorem}[{\cite[Theorem 1.2]{BrStab}, \cite{BMS}}]\label{thm_mnf_str}
    When not empty, the space $\stab(\calD)$ is a complex manifold of dimension $\rk \Lambda$,
    locally isomorphic to $\Hom_\bZ(\Lambda,\C)$ via the forgetful morphism
    \begin{equation}\label{forgetful_map}\cl Z: \sigma=(\cl P,Z)\mapsto Z.\end{equation}
\end{theorem}
Theorem \ref{thm_mnf_str} means that it is enough to deform the central charge in order to
cover any small neighbourhood of a stability condition in $\stab(\calD)$. In fact, its proof is based on the
deformation properties of the central charge, proved in \cite[\S 7]{BrStab}, that,
in turn, are guaranteed by the support property.
Some remarks are due.
The original request, for the space to be well-behaved, was referred to as
\lq\lq local-finiteness\rq\rq\ (Definition 5.7 in
\cite{BrStab}). It is implied by the support property appearing in the
currently accepted definitions, see \cite{KS,BMS}.
The local homeomorphism $\cl Z$ of \eqref{forgetful_map} showed in \cite{BrStab} was promoted to a local isomorphism
in \cite[Appendix A]{BMS}.
An alternative proof of Theorem \ref{thm_mnf_str} is given in the recommended
paper \cite{bayer_short} by A.\ Bayer.

\smallskip\par
We restrict to a connected component of the space $\stab(\calD)$.

\begin{lemma}[{\cite[Corollary 5.2]{woolf2}}] If $\sigma_1=(\calA_1,Z_1)$ and $\sigma_2=(\calA_2,Z_2)$ are in the same connected component of $\stab(\calD)$, then $\calA_1$ and $\calA_2$ are related by a finite sequence of forward or backward tilts at some (possibly trivial) torsion pairs.
\end{lemma}
This lemma becomes more concrete under some finiteness assumption on $\calD$.
Let $\cl H$ be the heart of a bounded t-structure on $\calD$. We denote by $\stab(\cl H)\subset\stab(\calD)$ the subset
consisting of stability conditions supported on
$\cl H$. Subsets $\stab(\cl H)$, as $\cl H$ varies, partition
$\stab(\calD)$. They need not be either open or closed. As
remarked in Subsection \ref{sec:stab:examples}, if $\cl H$ is
a finite heart, with
$\Sim(\cl H)=\{[S_1],\dots, [S_n]\}$, then any group homomorphism
$Z\in\Hom(K(\cl H),\C)$ such that $Z(S_i)\in\halfplane$
automatically satisfies the Harder--Narasimhan
condition and the support property.
Therefore $\stab(\cl H)\subset\stab(\calD)$ is isomorphic to $\bH^n$, and Proposition \ref{prop:chamberStab} below
describes how to \lq\lq glue\rq\rq\ such pieces.

\begin{prop}[{\cite[Section~5]{bridgeland_survey},
        \cite[Proposition~2.6,
        Corollary~2.10]{woolf1}}] \label{prop:chamberStab}
    Let $\calA_1$ be the heart of a bounded t-structure in $\calD$ and suppose that $\calA_1$ is finite. Let $S$ be a simple object in $\cl A_1$. If $\emptyset\neq\cl W_S\subset \stab(\cl A_1)$ is the real-codimension~$1$ locus for which a $S$ has phase $1$, and all other simples
    have phase in $(0,1)$, we have
    \[\stab(\cl A_1)\cap\overline{\stab(\cl A_2)}
    \= \cl W_S \Longleftrightarrow \cl A_2:=\mu^\flat_S\cl A_1.
    \]
    If $\cl W$ is the subset of $\stab(\cl A_1)$ of stability conditions such that
    $k$ simples $S_{i_1},\dots, S_{i_k}$ have phase 1 and the others have phase less
    than~1, we have
    \[\cl W\subseteq \stab(\cl A_1) \cap \overline{\stab(\cl A_2)} \Longleftrightarrow
    \cl A_2=\mu^\flat_\torsion\cl A_1
    \]
    for some torsion class $\torsion\subset\bra S_1,\dots, S_k\ket_{\cl A_1}$. The real
    dimension $\dim_\R \stab(\cl A_1) \cap \overline{\stab(\cl A_2)}$ is at least $k$.
\end{prop}
\begin{proof}The proof is based on the $\bC$-action defined below. The inclusion of $\cl W$ need not be an equality.
\end{proof}
The real-codimension 1 boundaries $\cl W_S$ of sets $\stab(\cl A)$ are sometimes called \emph{walls (of second type)}. The connected components of the complement of the closure of the union of walls in $\Stab(\calD)$ are often called \emph{chambers}.

Another type of \emph{wall and chamber} decomposition of
the stability manifold is given by so-called walls of
\emph{marginal stability}. They are the set
$\cl W_\alpha(\beta)$, where the central charge of
non-proportional classes $\alpha,\beta\in K(\calD)$
with non-trivial extension
satisfies $Z(\alpha)/Z(\beta)\in\R$. Along these walls,
phenomena of strict semistability may happen, and in fact,
they may identify regions of the stability manifold on
which the property of being stable of an object changes.

\subsection{Group actions}
A question is whether we can cover a whole
connected component of the stability space $\stab(\calD)$ by
starting at some known family of stability conditions and acting by a group. While
in general this is not true, there are several examples
where this strategy allows to compute
$\stab(\calD)$ or an appropriate quotient, usually $\stab(\calD)/G$ for a subgroup $G\subset\Aut(\calD)$.

There are two natural actions on $\stab(\calD)$ induced, respectively, by autoequivalences of the category and by the orientation-preserving transformation of $\bC$.

The group of autoequivalences $\Aut(\calD)$ acts
on $\stab(\calD)$ by isometries
\[\Phi.(\cl H, Z)= \left( \Phi(\cl H), Z\circ[\Phi]^{-1} \right),
\]
or, equivalently,
\[\Phi.(\cl P,Z)= \left( \{\Phi(\cl P(\phi))\}_{\phi\in\R},Z\circ[\Phi]^{-1}\right).
\]
Here $[\Phi]$ denotes the map induced by $\Phi\in\Aut(\calD)$ on the Grothendieck group
$K(\calD)$. Note that there is a special subgroup
$\Aut^0(\calD)\subset \Aut(\calD)$ consisting of auto-equivalences that induce the
identity on the Grothendieck group.
The $n$-th shift functor $[n]\in\Aut(\calD)$ acts by
\[ [n].(\cl H, Z)=(\cl H[n], (-1)^n Z).\]
The universal covering $\widetilde{\GL}^+(2,\R)$ of the group $\GL^+(2,\R)$ of $2\times 2$ matrices
with real entries and positive determinant, acts smoothly  on the right as follows.
We identify
\[\widetilde{\GL}^+(2,\R)=\left\{(f,T) \mid \begin{array}{l}
    f:\R \to \R \text{ increasing, } f(\phi+1)=f(\phi)+1, \\
    T: \R^2\to \R^2\in \GL^+(2,\R),\ T_{|\R^2/\R_{>0}} \equiv f_{|\R/2\Z}
\end{array} \right\},\]
and define the image of $\sigma=(\cl P,Z)$ under $(f,T)$ as
\[\left(\left\{\cl P(f(\phi))\right\}_{\phi\in\R} , T^{-1}\circ Z \right).\]
The $\Aut(\calD)$-action and the $\widetilde{\GL}^+(2,\R)$-action commute, and also commute with the free action by scalars
\begin{equation} \label{eq:Caction}
    \lambda.(\cl P,Z):=(\cl P', e^{-\pi i \lambda}Z), \quad \text{where }\cl P'(\phi)=\cl P(\phi + \Re \lambda)
\end{equation}
for $\lambda \in \bC$, which coincides with the action of $[n]\in \Aut(\calD)$ when $\lambda=n \in \bZ$.
Note that the $\bC$-orbits $\bC.\sigma=\{\lambda.\sigma\mid \lambda\in\bC\}$ are closed, and
the restriction of the metric $d$ to $\bC.\sigma$ is given by
\[d (\sigma, \lambda.\sigma)=\max\{|\Re\lambda|,\pi|\Im\lambda|\}.\]
The real part of $\lambda$ produces a modification that can be
pictorially described as a
\lq\lq rotation\rq\rq\ of the heart of the t-structure,
because $e^{-\pi i \lambda}Z$ identifies a different
distinguished half-plane
in the complex plane, or as a
translation of the heart of the slicing,
as it is affine on the set of
phases of semi-stable objects. In fact, we can regard $\C\subset \widetilde{\GL}^+(2,\R)$. The imaginary part of $\lambda$ is
responsible for the rescaling of
the central charge.
If $0<\Re\lambda\leq 1$, the $\bC$-action on a stability condition $\sigma$, represented according to Definition \ref{def_stab1} as $\sigma=(\calA, Z)$, gives
\[\begin{aligned}
    &\lambda.(\calA,Z)=(\mu^\sharp_{\torsionfree_\lambda}\calA, e^{-\pi i \lambda} Z),\\
    & (-\lambda).(\calA,Z)=(\mu^\flat_{\torsion_\lambda}\calA, e^{\pi i \lambda} Z),
\end{aligned}\]
where $\torsionfree_\lambda$ is the torsion-free class
\[\torsionfree_\lambda=\cl P((0,\Re\lambda])=\bra E\in\calA\mid E \text{ semistable}, \phi_\sigma(E)\leq \Re\lambda\ket,\]
and $\torsion_\lambda$ is the torsion class
\[\torsion_\lambda=\cl P((1-\Re\lambda,1])=\bra E\in\calA\mid E \text{ semistable}, \phi_\sigma(E)>\Re\lambda\ket.
\]
Note that the $\C$-action and the $\widetilde{\GL}^+(2,\R)$-action does not change the set of semistable objects, and the result on the slicing is
essentially a relabelling.
The space $\PP\stab(\calD):={\bC\backslash\stab(\calD)}$ is called the
\emph{projectivised stability space}. It is a (non-compact) complex projective manifold locally modelled
on $\PP\Hom(\Lambda,\C)$.

\subsection{Some questions regarding $\Stab(\calD)$}

Despite the attention that
Bridgeland stability conditions and the stability manifold
have attracted since their introductions, a general strategy
for constructing a stability structure
on
a triangulated category is not known yet. The definition problem appears usually when we deal with geometric categories, such as the bounded derived categories of complex varieties.
At the same time,
saying that the space $\stab(\calD)$ is empty might be even harder.

To my experience, there are two main research directions concerning
the stability manifold of a triangulated category from
representation theory.

\subsubsection*{Mirror symmetry} One direction arises in the context of
mirror symmetry and has to do with a theory of invariants counting
(in the appropriate sense) semi-stable objects, and with encoding
such invariants in some additional geometric structures, which are
analogous to other structures appearing in other moduli problems, especially Gromov--Witten theory.
These structures may involve pencils of isomonodromic connections on
the tangent bundle to the space $\stab(\calD)$, e.g.,
\cite{FGS_stabdata,barbieri2019frobenius,barbieri2019construction,tom_geomstr}. The enumerative theory associated
with Bridgeland stability conditions is often called
Donaldson--Thomas theory (in analogy with counting of sheaves
on a Calabi--Yau three-fold) or a theory of BPS indices (where this
notation is borrowed from physics). While such theories are not
completely developed yet, quiver categories provide examples to start with \cite{mozgovoy1}.

\subsubsection*{Classical questions}
The other direction has to do with computing the whole $\stab(\calD)$ and studying it as a topological and complex space. While this has a
more classical flavour, an interesting feature
of this complex space is that studying its topology
and geometry usually requires deep
understanding of bounded t-structures on $\calD$. Woolf, in the already mentioned paper \cite{woolf1}, explains relations between the topology
of $\Stab(\calD)$ and tilting, under suitable assumptions.

We usually
restrict to a connected component $\stab^\circ(\calD)$
of the space $\stab(\calD)$. We choose it
by
requiring that it contains stability conditions supported on a
fixed chosen heart of $\calD$.
In general, the space $\stab(\calD)$ might not be connected
(as shown in examples in \cite{sven} and in \cite{BMQS}), thought it is
often conjectured to be connected. Connectedness is proved in some
cases, e.g., for $\calD^b\left(\rep(\bullet\stackrel{n}{\Rightarrow}\bullet)\right)$ (including the derived category of $\C\PP^1$) in
\cite{macri_curves}, and for the derived categories of
coherent sheaves on the minimal resolutions of
$A_n$-singularities supported at exceptional sets (which also
admit a description in quivers terms) in
\cite{ishii2010stability}, to name some early examples.

Using homological tools and the study of bounded t-structures,
there are results about
simple connectedness and contractibility of connected components
of stability manifolds. Examples in this direction include \cite{pauksztello2018contractibility,adachi2019discreteness,QYW}.

On the other hand,
any non-empty connected component of a stability manifold is non-compact.
(Partial) compactifications of (a connected component of) the stability manifold or a quotient by the groups $\C$ or $\Aut(\calD)$ have recently been
proposed and studied in few classes of examples. In \cite{thurston_compact} the authors consider the closure of the image of an embedding of $\stab^\circ(\calD)/\C$ in a projective space, and define a \emph{Thurston compactification}. In \cite{uk_compact} and \cite{multiscale},
infinitesimal deformations of the mass function or the central charge are introduced in such a way that they induce stability conditions on appropriate triangulated quotients of $\calD$.
The two strategies lead to the notion of {\textit{lax}} \emph{stability conditions} and of \emph{multi-scale stability conditions}, respectively.

\subsection{Some examples}
We collect an incomplete list of references of computations of
stability manifolds, before focusing on one specific example
in the next section.

\subsubsection*{Stability conditions for geometric categories.} The stability
manifold of
$\mathcal{D}^b(\operatorname{Coh}C_g)$, where $C_g$ is a curve of genus
$g=0$ or
$g\geq 1$, was computed by Bridgeland and by Macr{\`\i} at very early
stages. They were followed by K3 surfaces (summarised in \cite{macri_ben}), and some Calabi--Yau three-folds, which are the natural target spaces of corresponding theories in
physics and the source of conjectures relating invariants
from different theories. Perhaps the most investigated Calabi--Yau three-fold was the quintic three-fold
$x_1^5+x_2^5+x_3^5+x_4^5+x_5^5\in\PP^4\C$,
which is an interesting variety from many points of view in mirror symmetry. It was completed only in 2018 in \cite{quintic} after great efforts.
Computing the stability manifold for the bounded derived category of a variety $X$ of dimension $3$ and higher
is complicated, see \cite{BMS, BLMS} and references in the introduction of \cite{liu}. A strategy to construct a stability condition by Bayer, Macr{\`\i}, Toda is
called tilt-stability \cite{BMT}. With this procedure, a weaker notion of stability is constructed on $\operatorname{Coh} X$, and deformed to induce an honest stability condition
on an appropriate heart.

\subsubsection*{Stability manifolds for derived and Calabi--Yau quiver categories.}
When we deal with quiver categories we can count on some amount of combinatorics,
and they therefore represent a good starting point for testing
conjectures related with Bridgeland theory. On the other hand, the study of their stability
manifold has sometimes revealed independently interesting features.
Some examples of
results concerning the stability manifold for quiver categories are
\cite{toms,ikeda,BS15,QYW,CW,caitlin}.

\subsubsection*{Stability manifolds related with finite-dimensional complex Lie algebras.}
Let $\mathfrak{g}_\Gamma$ be the complex Lie algebra associated with the
Dynkin quiver $\Gamma$. The spaces of stability conditions of certain
triangulated categories $\mathcal{D}_\Gamma$ associated with $\Gamma$ are
related with $\mathfrak{g}_\Gamma$ in a way that depends on the category and
its autoequivalences. The results involve isomorphisms between (a connected
component of)
a stability manifold and (the universal cover of) the quotient of a Cartan
subalgebra $\mathfrak{h}\subseteq\mathfrak{g}_\Gamma$
by a Weyl group. Some of these categories, and their stability manifolds, are
considered for
instance in \cite{kleinian,thomas,ishii2010stability,toms}. Beside this relation being interesting in itself and providing
different incarnations of the theory of Dynkin diagrams, it also provides an
example of stability manifolds enriched with additional and conjectured geometric
structure.

%%%%%%%%%%%%%%%%%%%%%%%%%%%%%%%%%%

\section{An example: the stability manifold of {$\CY_3$} categories from surfaces}\label{sec_ginzburgexample}
In this section we review the description of (a connected component) of the stability manifolds of a class of triangulated categories, defined in \ref{sec:ginzburg} below, which
is known thanks to the Bridgeland--Smith correspondence relating stability conditions and a class of meromorphic quadratic differentials. The idea behind this section is to emphasise some tools that might be useful in order to describe the stability manifold, and some fruitful interaction between two apparently very different moduli problems.

The Bridgeland--Smith correspondence consists of two parts. The first concerns the construction of a triangulated category from a marked bordered Riemann surface and the study of its finite hearts. This is summarised in Subsection \ref{sec:ginzburg}. The second is the isomorphism between two complex spaces: the stability manifold and a space of meromorphic quadratic differentials with simple zeroes.  After the main theorem (Theorem \ref{thm:BS15_iso}), we state few consequences, concerning the moduli space of stability conditions and the moduli space of quadratic differentials, respectively. A small
explicit example is provided in Subsection \ref{sec:A2} to clarify the correspondence. The last subsection briefly mentions some generalisations of the Bridgeland--Smith correspondence to other classes of quadratic differentials and triangulated categories.

We need some preliminary definitions, which are summarised in
Subsection \ref{sec:prelim}.

\subsection{Preliminaries}\label{sec:prelim}

\subsubsection*{Quivers with potential and associated (dg) algebras}\label{subsec:QP}

We denote by $(Q,W)$ a finite (possibly disconnected) oriented quiver $(Q_0,Q_1,s,t)$ that has no loops or 2-cycles, with finite set of vertices $Q_0$, set of arrows $Q_1$, source and tail functions $s,t$, and with potential $W$.
The completion of the path algebra $k Q$ with respect to the bilateral ideal generated by arrows in $Q_1$ is denoted by $\widehat{k Q}$.  The lazy path (of length 0) at the vertex $j\in Q_0$ is denoted by $e_j$. The potential $W$ is a formal sum of cycles in
$\widehat{k Q}$, up to cyclic equivalence, i.e., $\alpha_1\alpha_2\cdots\alpha_m$ with $t(\alpha_i)=s(\alpha_{i+1})$ for $i\in\bZ/m\bZ$, is equivalent to $\alpha_2\cdots\alpha_m\alpha_1$.
The cyclic derivative with respect to an arrow $a\in Q_1$ is the unique $k$-linear map that takes a cycle of the form $c=uav$ with $u,v\in \widehat{kQ}$, to $vu\in\widehat{kQ}$, and a cycle not containing $a$ to $0$. By $\partial W$ we denote the ideal $\bra \del_a W \mid a\in Q_1\ket\subset \widehat{k Q}$.
In the examples we are most interested in, the potential involves all basic cycles, and the ideal $\del W$ is generated by monomials consisting of at least two letters. See \cite{DWZ} for these basic notions.

%\begin{definition}
The \emph{Jacobian algebra} $\cl J(Q,W)$ of a quiver with potential is the quotient of the complete path algebra $\widehat{k Q}$ with respect to the ideal $\del W$.
%\end{definition}
We assume it is a finite dimensional algebra. The category of finitely generated modules over $\cl J(Q,W)$ is denoted by
\[\calA(Q,W):=\modules \cl J(Q,W),\]
or $\calA_Q$ for simplicity, and coincides with $\rep(Q,W)$,
the category of finite dimensional representations of $Q$ with
relations induced by the generators of $\del W$.
It is a finite-length, finite, abelian category; see, for instance, \cite[{Section 3}]{keller_dilog}.
The following example \eqref{qp_example} illustrates the relations induced by the potential and the resulting Jacobian algebra:
\begin{equation}\label{qp_example}
    \begin{aligned}
        &Q=\xymatrix{\bullet \ar[rr]^\alpha && \bullet \ar[dl]^\beta\\
            &\bullet \ar[ul]^\gamma
        }
        \qquad
        W=\alpha\beta\gamma,
        \\
        &\cl J(Q,W)=\widehat{kQ}/\bra \alpha\beta,\beta\gamma,\gamma\alpha\ket =  kQ/\bra \alpha\beta,\beta\gamma,\gamma\alpha\ket.
    \end{aligned}
\end{equation}
Part of the information of the category $\calA_Q$ is encoded in
the combinatorics of the quiver with potential. The finite set
of vertices $Q_0=\{1,\dots, n\}$ is in bijection with the set
$\Sim(\calA_Q)=\{[S_j]\,|\,j\in Q_0\}$ of (iso-classes of)
simple objects of $\calA_Q$. Their classes in the Grothendieck
group form a basis of primitive vectors of
\[K(\calA_Q)\simeq\Z^{|Q_0|}.\]
The dimension $\ext(S_i,S_j)$ of the extension group $\Ext_{\calA_Q}(S_i,S_j)$ is given by the number of arrows $q_{ij}$ from $i$ to $j$.
\medskip\par
A \emph{mutation~$\mu_i$} of $(Q,W)$ at a vertex $i$ is an operation that creates
a new quiver with potential $\mu_i(Q,W)=(\mu_i Q,\mu_i W)$ with the same set of
vertices. The new set of arrows $(\mu_i Q)_1$ is constructed from $Q_1$ as follows:
\begin{itemize}[nosep]
    \item[(1)] for any pair of arrows $a,b\in Q_1$, with $t(a)=i=s(b)$, add a new arrow $[ab]:s(a)\to t(b)$,
    \item[(2)] replace any arrow with source or target $i$ with the opposite arrow $a^*$,
    \item[(3)] remove any 2-cycle.
\end{itemize}
The new potential $\mu_i W$ is defined as $W'+W''$, where $W'$ is obtained by replacing $[ab]$ any composition $ab$ with $t(a)=i=s(b)$, and where $W''=\sum_{a,b}[ab]b^*a^*$. In example \eqref{qp_example}, the quiver with potential $(Q,W)$ is the mutation $\mu_2(A_3,0)$ of the linear oriented quiver $A_3=\bullet_1\to \bullet_2\to \bullet_3$ with trivial potential at vertex $2$.

A quiver with potential $(Q,W)$ is \emph{non-degenerate} if any quiver with potential obtained from $(Q,W)$ by iterated mutations
has no loops or 2-cycles. Given a non-degenerate quiver with
potential $(Q,W)$, we fix an integer $N\geq 3$. We assume that either $N=3$, or $N>3$ and $Q$ is acyclic.
\medskip\par
The $N$-th \emph{complete Ginzburg differentially graded (dg) algebra}
\[\Gamma_N(Q,W):=(\widehat{k \Qgrad},d)\]
is defined as follows \cite{ginzburg,K1}. First introduce the graded
quiver $\Qgrad$ with vertices $\Qgrad_0=Q_0$ and graded arrows:
\begin{itemize}[nosep]
    \item every $a:i\to j\in Q_1$, in degree 0,
    \item an opposite arrow $a^*:j\to i$ for any $a:i\to j\in Q_1$, in degree $-(N-2)$,
    \item a loop $e_i$ for any $i\in Q_0$, in degree $-(N-1)$.
\end{itemize}
Then the underlying graded algebra of $\Gamma$ is the completion $\widehat{k \Qgrad}$ of
the graded path algebra~$k\Qgrad$ with respect to the ideal
generated by the arrows of $\Qgrad$. Finally,
the differential $d$ of~$\Gamma$ is the unique continuous linear endomorphism,
homogeneous of degree~$1$, that satisfies the Leibniz rule and takes the
following values:
\[
da=0, \qquad d a^* \=  \partial_a W,
\qquad d e_i  \=  \sum_{a\in Q_1} \, e_i[a,a^*]e_i,
\]
where $e_i$ is the idempotent element at $i\in Q_0$ in $\bk Q$, i.e., $e_i^2=e_i$, and $e_i\gamma$ (resp., $\gamma e_i$) equals zero if $t(\gamma)\neq i$ (resp., $s(\gamma)\neq i$) and equals $\gamma$ otherwise.

\begin{rmk}\label{HGamma}The zero-th co-homology satisfies
    \[H^0\big(\Gamma_N(Q,W)\big)\simeq \cl J(Q,W).\]
\end{rmk}
\begin{proof}Assume $N=3$; then the arrows in $\bar{Q}$ are in degree 0 (original arrows), $-(N-2)=-1$ (opposite arrows), and $-(N-1)=-2$ (loops). By definition of $H^0=\frac{\ker d_0}{\Im d_1}$ we obtain that $H_0\big(\Gamma_N(Q,W)\big)= kQ/\del W$.
    If there is no potential, we simply note that $H_0(\Gamma_N(Q))=kQ$.
\end{proof}
\begin{definition} The \emph{perfectly valued derived category} of the dg algebra $\Gamma_N(Q,W)$ is the full triangulated subcategory of the (unbounded) derived category $\calD(\Gamma_N(Q,W))$ whose objects are dg modules of finite dimensional total cohomology. It is denoted by $\pvd(\Gamma_N(Q,W))$ or $\calD_{fd}(\Gamma_N(Q,W))$.
\end{definition}

Say $\Gamma_N:=\Gamma_N(Q,W)$. The graded quiver $\Qgrad$ is in fact the
$\Hom^\bullet$-quiver of $\calD(\Gamma_N)$ and
$\pvd(\Gamma_N)$, viewed as triangulated categories generated by the dg modules $e_i\Gamma_N$: the graded arrows
(with absolute value of the grading augmented by 1) $i\to j$ form a basis for
$\Hom^\bullet(e_i\Gamma_N, e_j\Gamma_N)$. The triangulated category $\pvd(\Gamma_N)$ is Calabi--Yau of dimension $N$, i.e., for any objects $E,F$,
there is a natural isomorphism  of $k$-vector spaces
$\Hom(E,F)\stackrel{\simeq}{\to}\Hom(F,E[N])^\vee$.

By Remark \ref{HGamma} and \cite{KY}, the derived category $\calD(\Gamma_N)$
has a t-structure with heart $\Mod\cl J(Q,W)$ that restricts to
$\pvd(\Gamma_N)$,
on which it defines a bounded t-structure with heart $\cl \modules\cl J(Q,W)$.

Quivers with potential that are related by a mutation at a vertex
define equivalent perfectly derived categories, so that we may
say that any quiver with potential which is mutation-equivalent
to $(Q,W)$ defines a bounded t-structure on $\pvd(\Gamma_N)$.
From this perspective, simple tilts with
respect to a simple object $S_i$ are a categorification of
mutations with respect to the $i$-th vertex at the level of
quivers.

%===================
\subsubsection*{Decorated marked surfaces} \label{subsec_dms}
%===================
We briefly
review here the definition of a weighted decorated
marked surface, which is an enhancement of the more classical
notion of a marked surface, via the choice of a set of internal points, each weighted by an integer.
For simplicity, we assume that there are no
internal marked points (\emph{punctures}) nor decorations of weight $-1,0$. Given a weighted decorated marked surface,
we can define a
mixed- or tri-angulation, to which a quiver with potential
will be attached.
The following description is far from exhaustive, and we
refer to \cite{LabaFragQP} for the original construction of
quivers with potential from triangulations of marked surfaces, to
\cite{qiu15} for the refinement to decorated marked
surfaces, and to \cite{BMQS} for the general definitions.

\begin{definition}
    A \emph{decorated marked surface} (without punctures) $(\surf,\MM, \Delta)$ consists of
    \begin{itemize}
        \item a connected differentiable Riemann surface $\surf$, with a fixed orientation and border $\partial\surf=\bigcup_{j=1}^b
        \partial_j$,
        \item a non-empty finite set
        $\MM$ of marked points on the boundary components, such that each connected component of
        $\partial\surf$ contains at least one marked point, and
        \item
        a non-empty finite set $\Delta= \{p_i\}_{i=1}^r$ of
        points in the  interior of~$\surf$.
    \end{itemize}
\end{definition}

Up to homeomorphism, $(\surf,\MM,\Delta)$ is determined by
the genus $g\geq 0$ of $\surf$, the number $b$ of boundary
components, the integer partition~$(M_j)_{j=1}^b$ of
the cardinality $|\MM|$ of $\MM$, where $M_j$
is the number of marked points on $\partial_j$, and the integer $r$.

A \emph{weight} function on $\Delta$ is a function
$\w\colon \Delta\to\Z_{\geq -1}$. Here we assume it takes values in $\Z_{\geq 1}$.
We say it is \emph{compatible} with $\surf$ and $\MM$ if
\begin{gather}\label{eq:cp1}
    \sum_{p\in\Delta} \w(p) - \sum_{j=1}^b(M_j+2) \=4g-4\,.
\end{gather}
If~$\w$, $\MM$, and $\surf$ are compatible, we will write $\sow$ for the class of $(\surf, \MM, \Delta, \w)$ up to diffeomorphism,
and call this tuple a \emph{weighted (decorated) marked surface}.
For simplicity, we will not distinguish between $\surf_\w$ and an underlying Riemann surface $\surf$.
\smallskip\par

We let $\sow^\circ:=\sow\setminus\partial\sow$.
An \emph{open arc} is an isotopy class of curves $\gamma\colon [0,1]\to\sow$
whose interior is in $\sow^\circ\setminus\Delta$ and whose endpoints
are in the set of marked points~$\MM$. An \emph{open arc system} $\{\gamma_i\}$ is a collection of open arcs on $\sow$ such that there is no \mbox{(self-)}intersection between any of them
in~$\sow^\circ\setminus\Delta$.

\begin{definition} A \emph{$\w$-mixed-angulation} is a maximal
    open arc system $\mathbb{A}$ which, \linebreak together with segments of the boundary components between any two marked points, tiles $\sow$ into polygons encircling a decoration of weight $w_i$ exactly if they have $w_i+2$ edges.
\end{definition}

The expression $\w$-\emph{mixed}-angulation insists on the fact that
polygons are allowed to have different shapes. The word
\emph{dissection} is used to indicate a similar maximal open
arc system, but in the context of classification of gentle
algebras (in particular the way a quiver is associated to a dissection, e.g., in \cite{palu}, looks different).
The simply decorated case, i.e., when $\w\equiv\mathbf{1}$, is studied, e.g., in \cite{qiubraid16,KQ2}. For this choice we write $\surfo$ for $\sow$. A triangulation~$\mathbb{T}$ of $\surf_\Tri$ is a $(\Tri)$-mixed-angulation, which in fact divides $\surf_\Tri$ into triangles, each containing exactly one decoration.
\medskip\par
The \emph{forward flip} at an arc $\gamma$ of
a $\w$-mixed-angulation is the operation that moves the
endpoints of $\gamma$ counter-clockwise along the adjacent open arcs of
the smallest polygon encircling $\stackrel{\circ}{\gamma}$ and
two decorations. The inverse operation is called
a \emph{backward flip}. These movements are \emph{relative to the decorations}, so that, for instance, performing twice a
forward flip at the same arc separating two triangles is not
the identity. They transform a $\w$-mixed-angulation into
another $\w$-mixed-angulation. See an example of a forward flip
of a triangulation of a simply decorated disc with five marked
points on the boundary in Figure \ref{fig_decorated_triang_A2}. The notion of forward and backward flips relative
to the decorations was originally proposed for triangulations
in \cite{KQ2}, and promoted to general $\w$ in \cite{BMQS}.

\begin{figure}[ht]
    \centering{
        \begin{tikzpicture}[scale=1.6]
            \useasboundingbox (-2.8,-1.3) rectangle (2.8,1.3);
            \begin{scope}[shift={(-2.5,0)}]
                \draw (0,0) circle (.6cm);
                \foreach \j in {0,...,4}{\draw (90+72*\j:.6cm) coordinate  (w\j);}
                \draw (.05,-.2) \ww;
                \draw (.4,0) \ww;
                \draw (-.3,.1) \ww;
                \draw[thick, blue] (w0) --  (w3);
                \draw[thick, teal] (w0) -- (w2);
                \foreach \j in {0,...,4}{\draw($(w\j)$)\nn;}
                \draw[teal] (-.3,-1) \nn ;
                \draw[blue] (.3,-1) \nn ;
                \draw[-stealth] (-.26,-1)--(.26,-1);
            \end{scope}

            \begin{scope}%[shift={(0,0)}]
                \draw (0,0) circle (.6cm);
                \foreach \j in {0,...,4}{\draw (90+72*\j:.6cm) coordinate  (w\j);}
                \draw (.05,-.2) \ww;
                \draw (.4,0) \ww;
                \draw (-.3,.1) \ww;
                \draw[thick, blue] (w0) --  (w3);
                \draw[thick, teal] (w1) .. controls (.5,.5) and (-.6,-.5) .. (w3);
                \foreach \j in {0,...,4}{\draw($(w\j)$)\nn;}
                \draw[teal] (-.3,-1) \nn ;
                \draw[blue] (.3,-1) \nn ;
                \draw[-stealth] (.26,-1)--(-.26,-1);
            \end{scope}

            \begin{scope}[shift={(2.5,0)}]
                \draw (0,0) circle (.6cm);
                \foreach \j in {0,...,4}{\draw (90+72*\j:.6cm) coordinate  (w\j);}
                \draw (.05,-.2) \ww;
                \draw (.4,0) \ww;
                \draw (-.3,.1) \ww;
                \draw[thick, blue] (w0) --  (w3);
                \draw[thick, teal] (w2) .. controls (-.55,2) and %.. (0,0);
                %\draw[thick, teal] (0,0) .. controls
                (.56,-2.5) .. (w0);
                \foreach \j in {0,...,4}{\draw($(w\j)$)\nn;}
                \draw[teal] (-.3,-1) \nn ;
                \draw[blue] (.3,-1) \nn ;
                \draw[-stealth] (-.26,-1)--(.26,-1);
            \end{scope}

            \begin{scope}[shift={(-2.5,0)}]
                \draw[-stealth,thick,darkgray] (1,0) -- (1.5,0);
            \end{scope}
            \begin{scope}
                \draw[-stealth,thick,darkgray] (1,0) -- (1.5,0);
            \end{scope}
        \end{tikzpicture}
    }
    \caption{Examples of triangulations of a simply decorated marked surface of genus $0$ with one boundary component and five marked points on the boundary (disc with five marked points), and of forward flips of the green arc. The red crosses denote the decorations. In all configurations, the resulting quiver with potential is a linear oriented quiver $A_2=\bullet \rightarrow \bullet$ with trivial potential $W_{A_2}=0$. }\label{fig_decorated_triang_A2}
\end{figure}
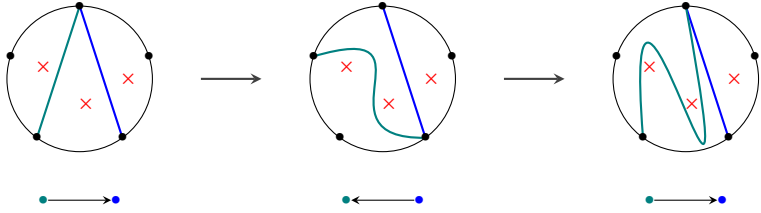

The \emph{exchange graph} $\Exch^\circ(\surf_\w)$ of a
weighted decorated marked surface is the directed graph whose
vertices are $\w$-mixed-angulations of $\surf_\w$ and whose
oriented edges are forward flips between them, containing a
vertex corresponding to a fixed initial triangulation.
It is an infinite graph.
\medskip
\par
Given a $\w$-mixed-angulation, we can define a quiver with
potential with the procedure described in Definition
\ref{quiver_from_mixedangulation} below. Note that there
are in fact different ways of constructing a quiver
(cf.\ \cite{BS15,nick,palu}).
\begin{definition}\label{quiver_from_mixedangulation}
    Given a $\w$-mixed-angulation $\mathbb{A}$ of a simply decorated marked surface, we associate to $\mathbb{A}$ a quiver with potential $(Q_\mathbb{A},W_\mathbb{A})$ with the following procedure:
    \begin{itemize} [nosep]
        \item the vertices of $Q_\mathbb{A}$ correspond to the open arcs in $\mathbb{A}$;
        \item the arrows of $Q_\mathbb{A}$ correspond to (clockwise) oriented intersection (at the endpoints) between open arcs in $\mathbb{A}$, so that there is a cycle in $Q_\mathbb{A}$ locally in each internal polygon;
        \item the potential $W_\mathbb{A}$ is the sum of all cycles that locally come from a polygon of $\mathbb{A}$ as above.
    \end{itemize}
\end{definition}

Note that the quiver $Q_{\mathbb{A}}$ has no loops nor two-cycles.

\subsubsection*{Spaces of quadratic differentials on marked surfaces}\label{sec:appendix:qdiff}
Let $\surf_g$ be a compact Riemann surface of genus $g$ and let
$\omega_{\surf_g}$ be its holomorphic cotangent bundle. A
\emph{meromorphic quadratic differential} $\Psi$ on $\surf_g$
is a meromorphic section of the line bundle
$\omega_{\surf_g}^2$. In a local complex coordinate $z$ on $\surf_g$,
it can be expressed as $\Psi(z)=f(z)dz\otimes dz$ for some
meromorphic function $f$. A meromorphic quadratic
differential $\Psi$ on $\surf_g$ has degree $4g-4$, which means that, if $p_i$ denotes the zeroes of $\Psi$ and
$q_j$ its poles, then
$\sum\operatorname{ord}_{\Psi}(p_i)-\sum \operatorname{ord}_{\Psi}(q_j)=4g-4$.
The book by Strebel \cite{strebel} is probably the best
reference for the
theory of quadratic differentials. We refer
nevertheless to \cite[Sections 2, 3, 4]{BS15} or
to \cite[Sections 3, 4]{BMQS} and references therein for
the main definitions  and for a quick introduction to the
moduli spaces of quadratic differentials appearing in this
survey, as well as for useful and explanatory pictures. We
refer to \cite{BS15} and \cite{KQ2} for (decorated) marked
Riemann surfaces and triangulations associated to a meromorphic
quadratic differential. We recall the most relevant notions
here in a rather heuristic way.
\begin{itemize}
    \item The critical profile of a meromorphic quadratic
        differential is the order vector of zeroes and poles
        $\big(\operatorname{ord}_\Psi(p_i), -\operatorname{ord}_\Psi(q_j)\big)_{i,j}$. We assume here that
        $\operatorname{ord}_\Psi(q_j)\geq 3$ for all $j$.
    \item Let $\mathbf{w}=(w_i)_{i=1}^r$ and $\mathbf{m}=(m_j)_{j=1}^b$ be t-uples of integers, with $w_i,m_j>0$. The moduli space of
        quadratic differentials (considered up to isomorphism) on a compact
        Riemann surface of
        genus $g$ and with critical profile $(\w,-\mathbf{m})$ is denoted by $\Quad_g(\w,-\mathbf{m})$.
    \item The standard double cover $(\widehat{\surf_g},\psi)$ of
        $(\surf_g,\Psi)$ is the data of $\pi:\widehat{\surf_g}\stackrel{2:1}{\to}{\surf_g}$ such that $\pi^*\Psi=\psi^2$. If $\widehat{P}$ and $\widehat{Q}$ are the preimages of the sets of zeroes and poles, respectively, of $\Psi$ under $\pi$, we let $\widehat{H}_1(\Psi)$ be the anti-invariant part of the relative homology group $H_1(\widehat{\surf_g}\setminus\widehat{Q},\widehat{P}; \C)$ with respect to the involution of $\widehat{\surf_g}$ associated to $\pi$. Integrating $\psi$ against a basis $\{\gamma_i\}_i$ of $\widehat{H}_1(\Psi)$ gives local coordinates $\int_{\gamma_i}\psi$ on the moduli spaces of quadratic differentials $\Quad_g(\w,-\mathbf{m})$. The map
        \[\int_*\psi: \widehat{H}_1(\Psi)\to \C, \quad \gamma\mapsto\int_\gamma\psi\]
        is the \emph{period map}.
    \item A horizontal trajectory is an integral curve for $\Psi$, i.e., a curve $\gamma\subset \surf$ where locally the quadratic
        differential has the form $dw^{\otimes 2}$, such that the
        imaginary part of $w\in \gamma$ is constant. The horizontal
        trajectories form the horizontal foliation, with distinguished
        trajectories connecting a zero and a pole, and generic
        trajectories connecting two poles.
    \item A saddle connection is (an isotopy class of) a straight arc connecting two
        zeroes along a fixed (arbitrary) direction with the horizontal trajectories, whose maximal
        domain is a finite interval. A quadratic differential is \emph{generic} if it has no
        \emph{horizontal} saddle connections, i.e., saddle connections along the horizontal direction. It means that there are not two zeroes aligned along the horizontal foliation.
    \item Near a pole $q_j$ of order at least $3$ on $\surf_g$, a quadratic differential
        $\Psi$ defines exactly
        $\operatorname{ord}(q_j)-2$ distinguished trajectories in the
        horizontal direction: those emanating from a zero. Around a
        zero $p_i$, there are exactly $\operatorname{ord}(p_i)+2$ of
        these distinguished trajectories. To see this write $\Psi$  (which, for instance, in the local complex coordinate $z$ centred at the zero $p_i$ behaves like $z^{\operatorname{ord}(p_i)}dz^{\otimes 2}$) in polar coordinates.
    \item Any generic meromorphic quadratic differential $\Psi$ on $\surf_g$, with $b$ poles
        of order $m_j\geq 3$ and $r$ zeroes $p_i$, of order
        $w_i\geq 1$ induces a weighted decorated marked surface
        $\surf_\w=(\surf,\MM,\Delta,\w)$ and a $\w$-mixed-angulation $\mathbb{A}$, which we describe.
        The surface $\surf$ is the real blow-up  of $\surf_g$ at the poles
        $$\surf=\operatorname{Bl}^\R_{q_1,\dots, q_b}\surf_g.$$
        This is a bordered surface obtained from $\surf_g$ by replacing
        any pole with a real dimension 1 boundary component that looks
        like $\R\PP^1$. The set of marked points $\MM$ is in bijection with the set of
        distinguished trajectories and is partitioned by
        $(M_j)_{j=1}^b=(m_j-2)_{j=1}^b$. See Figure \ref{picture_marked_pts}
        for an example.
        The set of decorations $\Delta$ coincides with the (preimage under blow-up
        of the) set of zeroes $\{p_1, \dots, p_r\}\subset \surf$ of $\Psi$
        endowed with a compatible weight function $\w$ defined by
        $w(p_i)=\operatorname{ord}_{\Psi}(p_i)$.

        A $\w$-mixed-angulation of $\surf_\w$ induced by $\Psi$ has edges which are isotopy classes of generic horizontal trajectories in $\surf$ minus the zeroes. See Figure \ref{fig_triang_disc}.
        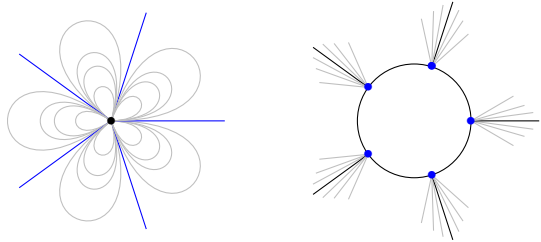
\begin{figure}[ht]
            \centering{
                \begin{tikzpicture}
                    \begin{scope}[shift={(-2,0)}, scale=1.5]
                        \foreach \j in {1,...,5}{
                            \draw[blue](72*\j:0) to (72*\j:1) ;
                            \draw[lightgray](2+72*\j:.01) .. controls (72*\j:1.5) and (72+72*\j:1.5) .. (70+72*\j:.01) ;
                            \draw[lightgray](2+72*\j:.01) .. controls (2+72*\j:1) and (70+72*\j:1) .. (70+72*\j:.01) ;
                            \draw[lightgray](4+72*\j:.01) .. controls (2+72*\j:.8) and (70+72*\j:.8).. (68+72*\j:.01) ;
                            \draw[lightgray](4+72*\j:.01) .. controls (2+72*\j:.5) and (70+72*\j:.5).. (68+72*\j:.01) ;
                        }
                        \draw (0,0) \nn;
                    \end{scope}
                    \begin{scope}[shift={(2,0)}, scale=1.5]
                        \draw (0,0) circle(.5);
                        \foreach \j in {1,...,5}{
                            \draw(72*\j:.5) to (72*\j:1.1);
                            \draw[lightgray](72*\j:.5) -- (4+72*\j:1.05) ;
                            \draw[lightgray](72*\j:.5) -- (8+72*\j:.98) ;
                            \draw[lightgray](72*\j:.5) -- (14+72*\j:.9) ;

                            \draw[lightgray](72+72*\j:.5) -- (68+72*\j:1.05) ;
                            \draw[lightgray](72+72*\j:.5) -- (64+72*\j:.98) ;
                            \draw[lightgray](72+72*\j:.5) -- (58+72*\j:.9) ;
                            \draw[blue](72*\j:.5) \nn;
                        }
                        \draw[blue](72:.5) \nn;
                    \end{scope}
                \end{tikzpicture}
                \caption{On the left: the horizontal trajectories around a pole of order $7$. The five distinguished trajectories are in blue. On the right: the corresponding five marked points on the real blow-up of the surface at the pole.}\label{picture_marked_pts}
            }
        \end{figure}
        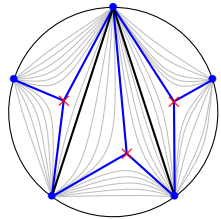
\begin{figure}[ht]
            \begin{tikzpicture}[scale=1.15, rotate=18]
                \draw (0,0) circle(1.2);
                \draw (-.5,.3) coordinate (p1);
                \draw (0,-.5) coordinate (p2);
                \draw (.7,-.1) coordinate (p3);

                \draw[lightgray] (72*1:1.2) ..controls (.1,-.5).. (72*4:1.2);
                \draw[lightgray] (72*1:1.2) ..controls (.2,-.5).. (72*4:1.2);
                \draw[lightgray] (72*1:1.2) ..controls (.3,-.5).. (72*4:1.2) ;
                \draw[lightgray] (72*1:1.2) ..controls (.6,-.1).. (72*4:1.2);
                \draw[lightgray] (72*1:1.2) ..controls (.5,-.1).. (72*4:1.2);

                \draw[lightgray] (72*1:1.2) ..controls (-.1,-.5).. (72*3:1.2);
                \draw[lightgray] (72*1:1.2) ..controls (-.3,-.5).. (72*3:1.2);
                \draw[lightgray] (72*1:1.2) ..controls (-.5,-.5).. (72*3:1.2) ;
                \draw[lightgray] (72*1:1.2) ..controls (-.5,.1).. (72*3:1.2);

                \draw[lightgray] (72*1:1.2) ..controls (-.5,.4).. (72*2:1.2);
                \draw[lightgray] (72*1:1.2) ..controls (-.5,.5).. (72*2:1.2);
                \draw[lightgray] (72*1:1.2) ..controls (-.5,.6).. (72*2:1.2);
                \draw[lightgray] (72*1:1.2) ..controls (-.5,.7).. (72*2:1.2);
                \draw[lightgray] (72*1:1.2) ..controls (-.5,.8).. (72*2:1.2);
                \draw[lightgray] (72*1:1.2) ..controls (-.5,.9).. (72*2:1.2);

                \draw[lightgray] (72*2:1.2) ..controls (-.6,.3).. (72*3:1.2);
                \draw[lightgray] (72*2:1.2) ..controls (-.7,.3).. (72*3:1.2);
                \draw[lightgray] (72*2:1.2) ..controls (-.8,.3).. (72*3:1.2);
                \draw[lightgray] (72*2:1.2) ..controls (-.9,.3).. (72*3:1.2);
                \draw[lightgray] (72*2:1.2) ..controls (-1,.3).. (72*3:1.2);

                \draw[lightgray] (72*3:1.2) ..controls (0,-.6).. (72*4:1.2);
                \draw[lightgray] (72*3:1.2) ..controls (0,-.7).. (72*4:1.2);
                \draw[lightgray] (72*3:1.2) ..controls (0,-.8).. (72*4:1.2);
                \draw[lightgray] (72*3:1.2) ..controls (0,-.9).. (72*4:1.2);
                \draw[lightgray] (72*3:1.2) ..controls (0,-1).. (72*4:1.2);

                \draw[lightgray] (72*4:1.2) ..controls (.7,-.3).. (72*5:1.2);
                \draw[lightgray] (72*4:1.2) ..controls (.7,-.4).. (72*5:1.2);
                \draw[lightgray] (72*4:1.2) ..controls (.7,-.6).. (72*5:1.2);
                \draw[lightgray] (72*4:1.2) ..controls (.9,-.6).. (72*5:1.2);
                \draw[lightgray] (72*4:1.2) ..controls (.9,-.7).. (72*5:1.2);

                \draw[lightgray] (72*5:1.2) ..controls (.7,.1).. (72*1:1.2);
                \draw[lightgray] (72*5:1.2) ..controls (.7,.2).. (72*1:1.2);
                \draw[lightgray] (72*5:1.2) ..controls (.7,.4).. (72*1:1.2);
                \draw[lightgray] (72*5:1.2) ..controls (.7,.6).. (72*1:1.2);
                \draw[lightgray] (72*5:1.2) ..controls (.7,.).. (72*1:1.2);

                \draw[blue,thick] (72*1:1.2) to (p1);
                \draw[blue,thick] (72*2:1.2) to (p1);
                \draw[blue,thick] (72*3:1.2) to (p1);
                \draw[blue,thick] (72*1:1.2) to (p2);
                \draw[blue,thick] (72*3:1.2) to (p2);
                \draw[blue,thick] (72*4:1.2) to (p2);
                \draw[blue,thick] (72*4:1.2) to (p3);
                \draw[blue,thick] (72*5:1.2) to (p3);
                \draw[blue,thick] (72*1:1.2) to (p3);	

                \draw[thick] (72*1:1.2) to (72*4:1.2);
                \draw[thick] (72*1:1.2) to (72*3:1.2);

                \draw[thick] (-.5,.3) \ww coordinate (p1);
                \draw[thick] (0,-.5) \ww coordinate (p2);
                \draw[thick] (.7,-.1) \ww coordinate (p3);
                \foreach \j in {1,...,5}{
                    \draw[blue](72*\j:1.2) \nn ;
                }
            \end{tikzpicture}
            \caption{In black, a triangulation on the disk  induced by (the horizontal foliation of) a quadratic differential on a genus zero surface, i.e., $\C\PP^1$, with critical profile $(1,1,1,-7)$. The red crosses correspond to the zeroes, and the blue lines are distinguished trajectories. The boundary component replaces the pole of the differential.
            }\label{fig_triang_disc}
        \end{figure}
    \end{itemize}
Fix a finite rank free abelian group $\Lambda$ and a reference weighted decorated marked surface $\surf_\w$.
\begin{itemize}
    \item A quadratic differential in $\Quad_g(\w,-\mathbf{m})$ is
    \emph{period-framed} or \emph{$\Lambda$-framed} if it is
    endowed with an isomorphism $\widehat{H}_1(\Psi)\simeq\Lambda$, so that we can define period coordinates valued in $\C^n=\Hom(\Lambda,\C)$. It is said to be
    \emph{Teichm\"uller-framed} if it is
    equipped with a
    diffeomorphism $\surf_\w\to\operatorname{Bl}^\R_{q_1,\dots, q_r}\surf_g$ preserving the
    marked points, the decorations, and their weights, up to
    diffeomorphism.
\end{itemize}

\subsection{$\CY_3$-Ginzburg categories from simply decorated marked surfaces}\label{sec:ginzburg}

Ginzburg categories are triangulated categories of
Calabi--Yau dimension~$N\geq 3$, associated with the
appropriately graded version $\bar{Q}$ of a quiver with potential $(Q,W)$ that must be acyclic if we chose $N\neq 3$.
Here we
are interested in a sub-class of Ginzburg categories
that have Calabi--Yau dimension~$3$ and are obtained from a
quiver with potential dual to a triangulation of a simply decorated (unpunctured) marked
Riemann surface
\[\surf_\Tri=(\surf,\MM, \Delta,\Tri),\]
of genus $g$, in the sense of Section \ref{sec:prelim}.
\medskip

Simple weights are particularly nice for several reasons. First, we notice that, once~$g$ and a partition $(M_j)_{j=1}^b$ of the cardinality~$|\MM|$ of~$\MM$ are fixed, the compatibility
condition \eqref{eq:cp1} fixes the number of
decorations in the interior of a surface~$\surf$
underlying~$\mathbf{S}_\Tri$. The choice of writing the
whole t-uple is aimed to remark that a set of decorations
has been fixed, and flips of arcs are relative to the
decorations. Moreover, forgetting the decorations, forward and backward flips coincide and are involutions. Un-decorated flips of arcs and mutations of quivers are in correspondence. We denote by $\EG(\surf, \MM)$ the finite exchange graphs whose vertices are un-decorated
triangulations $(\surf,\MM)$ and whose (un-oriented edges)
are flips (not relative to decorations).

Given a triangulation~$\mathbb{T}$ of~
$\mathbf{S}_\Tri$ separating the decorations, we denote by~$(Q_\mathbb{T},W_\mathbb{T})$ the
quiver constructed accordingly to Definition \ref{quiver_from_mixedangulation}, and we
consider the Ginzburg algebra~
$\Gamma_3(Q_\mathbb{T},W_\mathbb{T})$, defined in Section \ref{subsec:QP}. The $\Hom^\bullet$-quiver $\bar{Q}_\mathbb{T}$ of the derived category of~$\Gamma_3(Q_{\mathbb{T}},W_{\mathbb{T}})$, in this special case, can be read off from~$\mathbb{T}$ similarly to~$(Q_{\mathbb{T}},W_{\mathbb{T}})$. Its vertices are the arcs of the triangulations, and there is an arrow in degree $-i$ connecting two arcs $k_1$ and $k_2$ precisely if, walking along the perimeter of the triangle they belong to, there are exactly $i$ different arcs separating them. This means that, for each vertex of $\bar{Q}_0$, there is a loop in degree $-2$, and for each degree $0$ arrow from $k_1$ to $k_2$, there is a degree $-1$ arrow from $k_2$ to $k_1$.
\begin{figure}[ht]
    \centering{
        \begin{tikzpicture}[scale=1.6]
            \begin{scope}[shift={(-1.2,0)}]
                \draw (0,0) circle (.6cm);
                \foreach \j in {0,...,4}{\draw (90+72*\j:.6cm) coordinate  (w\j);}
                \draw[thick, blue] (w0) --  (w3);
                \draw[thick, teal] (w0) -- (w2);
                \foreach \j in {0,...,4}{\draw($(w\j)$)\nn;}
            \end{scope}

            \begin{scope}[shift={(1.2,0)}]
                \draw[teal] (-.6,0) \nn ;
                \draw[blue] (.6,0) \nn ;
                \draw (0,.15) node[font=\tiny,above] {0};
                \draw (0,-.15) node[font=\tiny,below] {-1};
                \draw[-stealth] (-.55,.03) ..controls (0,.2) .. (.55,.03) ;
                \draw[-stealth] (.55,-.03) ..controls (0,-.2) .. (-.55,-.03) ;
                \draw (-1.15,0)node[font=\tiny,right] {-2};
                \draw (1.15,0)node[font=\tiny,left] {-2};
                \draw[-stealth] (-.64,.03) ..controls (-1,.4) and (-1,-.4) .. (-.64,-.03) ;
                \draw[-stealth] (.64,.03) ..controls (1,.4) and (1,-.4).. (.64,-.03) ;
            \end{scope}
        \end{tikzpicture}
    }
    \caption{Graded quiver associated with a triangulation of a disk with five marked points on the boundary. The zero degree part is~$Q_{\mathbb{T}}$.}
\end{figure}
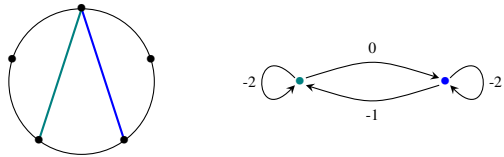

We fix an initial triangulation $\mathbb{T}^\circ$ and we define the triangulated $\CY_3$ category
\be \calD_3(\mathbf{S}_\Tri):=\pvd \Gamma_3(Q_{\mathbb{T}^\circ},W_{\mathbb{T}^\circ}),
\ee
which we call the Ginzburg category associated with the data of~$\surf_\Tri$ and~$\mathbb{T}^\circ$. By definition, the category $\calD_3(\mathbf{S}_\Tri)$ admits a bounded t-structure with finite heart~
$\cl H:= \modules\cl J(Q_{\mathbb{T}^\circ},W_{\mathbb{T}^\circ})$, which we call
\emph{standard} t-structure (and standard heart). The following theorem is a consequence of results in \cite{KQ2,KY}.

\begin{theorem} The category~$\calD_3(\mathbf{S}_\Tri)$ is $\Hom$-finite and has Calabi--Yau dimension~3. Different initial triangulations in~$\Exch^\circ(\surf_\Tri)$ define equivalent triangulated categories.
\end{theorem}

We consider the full exchange graph
$\EG(\calD_3(\mathbf{S}_\Tri))$ and we restrict to the
connected component containing the vertex
corresponding to the
standard heart $\cl H$. We denote it by
$\EG^\circ(\calD_3(\mathbf{S}_\Tri))$.
Let $\operatorname{sph}(\cl H)$ be the spherical twist group
of $\calD_3(\mathbf{S}_\Tri)$, i.e., the subgroup of
$\Aut\calD_3(\mathbf{S}_\Tri)$ generated by the set of spherical twists
$\{\Phi_S\,|\, [S]\in \Sim\cl H\}$ defined by
\[ \Phi_S(X)= \operatorname{Cone}(S\otimes \Hom^\bullet(S,X)\to X).	\]
The following theorem relates the exchange graph of $\calD_3(\mathbf{S}_\Tri)$ and that of the underlying decorated or un-decorated surface. %was proved in \cite{qiubraid16,qiu2018decorated} and builds on previous work of Labardini-Fragoso and Keller-Nic\'olas (see \cite{keller,kq}).

\begin{theorem}[{\cite{qiubraid16,qiu2018decorated}}]
\label{thm_iso_EG}
    There are isomorphisms of infinite exchange graphs
    \be \EG^\circ(\calD_3(\mathbf{S}_\Tri)) \simeq \Exch^\circ(\mathbf{S}_\Tri),
    \ee
    and of unoriented finite exchange graphs
    \be
    \EG^\circ(\calD_3(\mathbf{S}_\Tri))/\operatorname{sph}(\cl H) \simeq \EG(\surf,\MM).
    \ee
\end{theorem}

Beside stating a bijection between the set~$|\EG^\circ(\calD_3(\mathbf{S}_\Tri))|$ of (reachable) finite hearts of
$\calD_3(\surf_\Tri)$ and triangulations of~$\surf_\Tri$,
Theorem \ref{thm_iso_EG} proves a correspondence between the operation of forward (backward) simple tilt on bounded t-structures of~$\calD_3(\surf_\Tri)$ and forward (backward) mutations of arcs (relative to a set of decorations) on~$\surf_\Tri$.

\subsection{Stability conditions as quadratic differentials}

We fix $\surf_\Tri$ as above and an initial triangulation $\mathbb{T}^\circ$ whose dual quiver with potential $(Q_{\mathbb{T}^\circ},W_{\mathbb{T}^\circ})$ has set of vertices $Q_0:=(Q_{\mathbb{T}^\circ})_0$. We let $\Lambda:=\bZ^{|Q_0|}$.

In analogy with the notation used for the exchange graph $\EG^\circ(\calD_3(\mathbf{S}_\Tri))$ of $\calD_3(\surf_{\Tri})$, the symbol $^\circ$ in
$\Stab^\circ(\calD_3(\surf_{\Tri}))$ identifies the
connected component (principal part) of the stability manifold
$\Stab\calD_3(\surf_{\Tri})$ containing stability conditions
supported on the standard heart $\cl H= \modules \cl J(Q_{\mathbb{T}^\circ},W_{\mathbb{T}^\circ})$, while on groups of
autoequivalences of the category it identifies the subgroup
of those that preserve the principal part. The subscript
$_K$ on groups of autoequivalences refers to functors that
moreover act as the identity on the Grothendieck group. The
groups $\mathpzc{Aut}^\circ$ and $\mathpzc{Aut}^\circ_K$ are
the quotients of $\Aut^\circ(\calD_3(\surf_\Tri))$ and
$\Aut^\circ_K(\calD_3(\surf_\Tri))$ by the corresponding
subgroups of negligible autoequivalences, i.e., those that
act trivially on $\Stab^\circ(\calD_3(\surf_\Tri))$.
We act on the stability manifold by groups of autoequivalences
on the right, changing the convention from Section \ref{sec_stab_mfd}.
The forgetful map defined in \eqref{forgetful_map}, and here
restricted to $\Stab^\circ(\calD_3(\surf_\Tri))$, is $\cl Z:K(\calD_3(\surf_\Tri))\simeq \Gamma\to \C$.

In the next theorem, $\Quad_g(1^r,-\mathbf{m})$,
for $\mathbf{m}=(m_j)_{j=1}^b$, denotes the space of meromorphic quadratic differentials on a Riemann surface of genus $g$ with simple zeroes and poles of order $m_j$. Recall that quadratic differentials can be framed in several ways:
$\Quad^{\Lambda,\circ}_g(1^r,-\mathbf{m})$ denotes the relevant
connected component of the space of \mbox{$\Lambda$-framed} quadratic
differentials, while $\FQuad^\circ(\surf_{\Tri})$ denotes the
relevant connected component of the space of Teichm\"uller
framed quadratic differentials. In both cases the connected
component is specified by the choice of the triangulation $\mathbb{T}^\circ$ of $\surf_\Tri$. The period map $\int_*-$ was
defined in Section \ref{sec:prelim}.

\begin{theorem}[Bridgelan--Smith correspondence]\label{thm:BS15_iso}
    There is an isomorphism  of complex manifolds that fits
    into a commutative diagram
    \be\label{BSafterKQ}\xymatrix{
        K: \FQuad^\circ(\surf_{\Tri}) \ar[rr]^{\simeq} \ar[dr]_{\int} && \Stab^\circ(\calD_3(\surf_{\Tri}))\ar[dl]^{\cl Z}\\
        & \Hom(\Lambda,\C)
    }
    \ee
    and is equivariant with respect to the action of the
    mapping class group
    $\MCG(\surf_{\Tri})$ on the domain and of the group
    $\mathpzc{Aut}^\circ(\calD)$ on
    the range. It descends to  isomorphisms of complex orbifolds
    \be\label{kgamma}
    K^\Lambda: \Quad^{\Lambda,\circ}_g(1^r,-\mathbf{m}) \to
    \Stab^\circ(\calD_3(\surf_{\Tri}))/\mathpzc{Aut}^\circ_K(\calD_3(\surf_{\Tri})),
    \ee
    \be \label{kbar}
    \overline{K}: \Quad_g(1^r,-\mathbf{m}) \to \Stab^\circ(\calD_3(\surf_{\Tri}))/\mathpzc{Aut}^\circ(\calD_3(\surf_{\Tri})) \,.
    \ee
\end{theorem}

Defining explicitly the isomorphisms is beyond the scope of these notes, and we limit ourselves to a panoramic view. However, we refer the interested reader to the Introduction to \cite{BS15} for full understanding of the correspondence.

The original Bridgeland--Smith correspondence is about the
existence of the map $K^\Lambda$ and is the content
of \cite[Theorem 11.2]{BS15} proved in Section 11 of
op.\ cit. It was inspired by the work of the physicists
Gaiotto, Moore, and Neitzke. In fact, a version of equation \eqref{kgamma} holds
more widely for Ginzburg categories of Calabi--Yau
dimension $3$ associated with
quivers with potential from a triangulation of a marked surface
possibly with punctures (with few exceptions, listed in
\cite[Definition 9.3]{BS15}; see also Section 11.6),
at the cost of possibly replacing the
space $\Quad^{\Lambda,\circ}_g(1^r,\mathbf{m})$ with an
appropriate bigger orbifold described in
\cite[Section 6]{BS15}. The construction of the quiver and its
category, in the presence of punctures, is not considered here
to avoid technicalities.
The construction of the map $K^\Lambda$ in \cite{BS15} relies on
previous results by Labardini-Fragoso \cite{LabaFragQP} on a correspondence between flips of arcs
in the \emph{un-decorated} surface and mutations of the quiver.
This explains
why the result is up to to the action of the group $\mathpzc{Aut}^\circ_K(\calD_3(\surf_\Tri))$.

The lift $K$ of equation \eqref{BSafterKQ} was constructed in
\cite[Theorem 4.13]{kq}, where the combinatorial description of the category $\calD_3(\surf_\Tri)$ defined from a quiver with potential is enhanced to data from an arc system on a simply decorated marked surface,
and the operation of mutation of quivers is promoted to flips of arcs relative
to decorations, cf.\ Theorem \ref{thm_iso_EG}.

Last, the quotient $\overline{K}$ is added in \cite{BMQS}, where the reader can also find a more technical but still concise sketch of the proof of the whole theorem. In fact, the correspondence is stated here as it appears in \cite[Theorem 7.1]{BMQS}.

In the rest of the subsection we recall some consequences,
already emphasised in \cite{BS15}, that are intimately related with the construction of the isomorphisms of Theorem~\ref{thm:BS15_iso}, and present a simple example.

\subsubsection*{The exchange graph is a skeleton}
The main idea behind the correspondence is that a \emph{generic} configuration of open arcs (a triangulation)
on the decorated surface $\surf_\Tri$ singles out a finite
bounded t-structure on $\calD_3(\surf_\Tri)$ and flipping (isotopy classes of) arcs behaves like simple
tilts of hearts. A non-generic configuration induced by a non-generic meromorphic quadratic differential can be obtained \lq\lq by rotation\rq\rq\ of the differential, or by a continuous deformation of the position of the zeroes.
A consequence of this correspondence is that the exchange graph $\EG^\circ(\calD_3(\surf_\Tri))$,
is a \lq\lq skeleton\rq\rq\ for the space $\stab^\circ(\calD_3)$, which is \emph{tame} or \emph{generically finite}.
\begin{cor}\label{cor_tame} The space $\Stab^\circ(\calD_3(\surf_{\Tri}))$ is \emph{tame} or \emph{generically finite}, i.e.,
    \[\Stab^\circ(\calD_3(\surf_\Tri))=\C\cdot \bigcup_{\cl H \in |{\EG^\circ}|}\stab\cl H,\]
    where $|{\EG^\circ}|$ stands for the set of vertices of $\EG^\circ(\calD_3(\surf_\Tri))$.
\end{cor}
In fact, the isomorphisms of Theorem \ref{thm:BS15_iso} are first constructed
on the generic
locus of meromorphic quadratic differentials with no horizontal saddle connections that correspond to generic stability conditions in $\stab^\circ(\calD_3(\surf_\Tri))$ supported on a finite heart, and
without
strictly semistable objects. Then the maps are
extended to the whole spaces by geometric
arguments, so that the unnecessity of studying
other hearts for computing the space
$\stab(\calD_3(\surf_\Tri))$ comes as a consequence of the
isomorphism.

Last, the isomorphism of Theorem~\ref{thm:BS15_iso} also implies that the sets of stability conditions supported on
non-finite hearts have real co-dimension at least $1$ in $\stab(\calD_3(\surf_\Tri))$. This is the case, for instance, of the
subset of stability conditions supported on the \mbox{$\operatorname{Coh}\PP^1$-shaped} heart of the Ginzburg category associated with the Kronecker quiver,
as we expect.

\subsubsection*{The period map}
We focus again on the generic locus of spaces of
differentials. Here, saddle connections that are dual to edges of
triangulations
correspond to simple objects in the finite heart
$\cl H$
corresponding to the triangulation. On the Riemann surface,
we restrict to the space lying between the special
trajectories connecting two zeroes and two poles.
The choice of an orientation of the surface guarantees that the angle ``measured by the differential''
between a saddle connection $\gamma$ connecting two zeroes and a generic horizontal trajectory connecting two
poles is between
$0$ and $\pi$, and hence that $\int_{\gamma}\sqrt{\psi}\in\mathbb{H}$ for $\gamma\in\Gamma$. Identifying $\widehat{H}_1(\Psi)\simeq \Gamma\simeq K(\cl H)$, the period map can be interpreted as a central charge $Z(\gamma)=\int_\gamma\sqrt{\Psi}$.

\subsubsection*{Counting semistable objects}
The proof of isomorphism \eqref{kgamma} by Bridgeland and
Smith also provides a correspondence between saddle
connections of a generic quadratic differential and (iso-classes of) stable objects of the corresponding stability condition. These,
in turn, can be encoded in moduli spaces of stable
representations of finite-dimensional algebras (in the abelian sense mentioned in \ref{sec:stab:examples}), thanks to
the work of King \cite{king}, and hence enumerated in
appropriate sense. This opens new perspectives in
classification and counting problems in the theory of flat
surfaces. See \cite[{Theorem 1.4 and Section 1.6}]{BS15} for more details, and \cite{nick} for a more recent perspective.
Note that here we specify \lq\lq generic\rq\rq\ differential, i.e., we are not admitting counting strictly semistable objects. Note also that, at a
triangulated level, the notion of enumerative
invariants, when defined, often requires Calabi--Yau
dimension 3.

\subsection{$A_2$ example}\label{sec:A2} As an example, we explicitly work out the ingredients of the correspondence in the
$A_2$ case. This is far from an exhaustive model given that all hearts of bounded \mbox{t-structures} are finite and appear in $\EG(\calD_3(A_n))$. Let $\calD_3(A_2)$ be the Ginzburg
category of Calabi--Yau dimension $3$ associated with the linear $A_2$ quiver
\[\bullet_1\rightarrow\bullet_2\,.\]
The standard heart $\cl H_0=\rep(A_2)$ has two non-isomorphic simple objects
denoted by $S_1,S_2$ that are generators of the category, finitely many iso-classes of indecomposables, and five torsion classes. Let $E$ be an indecomposable in $\cl H_0$ fitting into the sort exact sequence $S_2\to E \to S_1$. The procedure of simple tilts gives rise to the following (partial) exchange graph \eqref{pentagon_hearts}, which is also the fundamental domain of $\EG(\calD_3(A_2))$ with respect to the action of the spherical twist group $\operatorname{sph}(\calD_3(A_2))$. Any $\cl H_i$, for $i=1,\dots, 4$, still has finitely many torsion pairs and two simple
generators, so
that two arrows should emanate from any $\cl H_i$ in
the full exchange graph.
\begin{equation}\label{pentagon_hearts}
    \xymatrix{
        &\cl H_1=\bra S_1[1],E\ket \ar[r]^{\mu^\sharp_E} & \cl H_2=\bra S_2,E[1]\ket \ar[dd]^{\mu^\sharp_{S_2}}\\
        \cl H_0=\bra S_1, S_2\ket \ar[ur]^{\mu^\sharp_{S_1}}\ar[dr]_{\mu^\sharp_{S_2}} \\
        &\cl H_3=\bra S_1, S_2[1]\ket \ar[r]_{\mu^\sharp_{S_1}} & \cl H_4=\bra S_1[1], S_2[1]\ket
    }
\end{equation}
The hearts $\cl H_i$, for $i=0,\dots, 4$, in \eqref{pentagon_hearts} are in fact with intermediate hearts with respect to $\cl H_0$, and in bijection with the classes of hearts supporting stability conditions in~$\stab^\circ(\calD_3(A_2))/\sph(\calD_3(A_2))$.

The quiver $A_2$ can be obtained, with the procedure described in Definition \ref{quiver_from_mixedangulation}, by a triangulation of the disc $\mathbf{D}$ with one
boundary component and five marked points $\MM$ on it. The interior
of the disc will contain three simply decorated points. So $\mathbf{D}_\Tri$ is specified by $\mathbf{D}=\operatorname{Bl}^\R_{\infty}\C\PP^1$, $b=1$ and $|\MM|=5$, and $\Delta=\{u_1,u_2,u_3\}$ with $\mathbf{w}=(1,1,1)$.
Figure \ref{decorated_flip} gives a pictorial explanation of
the second part of Theorem \ref{thm_iso_EG},
relating $\EG(\calD_3(A_2))/\sph(\cl H_0)$ and $\EG(\mathbf{D},\MM)$.
Compare it with the notion of forward flip from
Figure \ref{fig_decorated_triang_A2}.

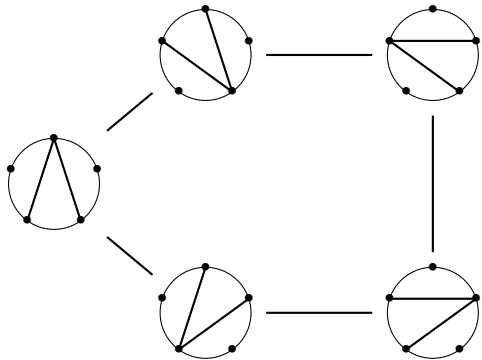
\begin{figure}[ht]
    \centering{
        \begin{tikzpicture}

            \begin{scope}
                \draw (0,0) circle (.6cm);
                \foreach \j in {0,...,4}{\draw (90+72*\j:.6cm) coordinate  (w\j);}
                \draw[thick] (w0) -- (w3);
                \draw[thick] (w0) -- (w2);
                \foreach \j in {0,...,4}{\draw($(w\j)$)\nn;}
            \end{scope}

            \begin{scope}[shift={(2,1.7)}]
                \draw (0,0) circle (.6cm);
                \foreach \j in {0,...,4}{\draw (90+72*\j:.6cm) coordinate  (w\j);}
                \draw[thick] (w0) -- (w3);
                \draw[thick] (w1) -- (w3);
                \foreach \j in {0,...,4}{\draw($(w\j)$)\nn;}
            \end{scope}

            \begin{scope}[shift={(2,-1.7)}]
                \draw (0,0) circle (.6cm);
                \foreach \j in {0,...,4}{\draw (90+72*\j:.6cm) coordinate  (w\j);}
                \draw[thick] (w4) -- (w2);
                \draw[thick] (w0) -- (w2);
                \foreach \j in {0,...,4}{\draw($(w\j)$)\nn;}
            \end{scope}

            \begin{scope}[shift={(5,1.7)}]
                \draw (0,0) circle (.6cm);
                \foreach \j in {0,...,4}{\draw (90+72*\j:.6cm) coordinate  (w\j);}
                \draw[thick] (w1) -- (w4);
                \draw[thick] (w1) -- (w3);
                \foreach \j in {0,...,4}{\draw($(w\j)$)\nn;}
            \end{scope}

            \begin{scope}[shift={(5,-1.7)}]
                \draw (0,0) circle (.6cm);
                \foreach \j in {0,...,4}{\draw (90+72*\j:.6cm) coordinate  (w\j);}
                \draw[thick] (w1) -- (w4);
                \draw[thick] (w2) -- (w4);
                \foreach \j in {0,...,4}{\draw($(w\j)$)\nn;}
            \end{scope}

            \draw[thick] (.7,.7) -- (1.3,1.2);
            \draw[thick] (.7,-.7) -- (1.3,-1.2);
            \draw[thick] (2.8,1.7) -- (4.2,1.7);
            \draw[thick] (2.8,-1.7) -- (4.2,-1.7);
            \draw[thick] (5,.9) -- (5,-.9);

    \end{tikzpicture}}
    \caption{Un-decorated triangulations and flips of the disc with five marked points.}
    \label{decorated_flip}
\end{figure}

A quadratic differential that induces the decorated marked surface $\mathbf{D}_\Tri$
with the procedure described in Subsection \ref{sec:prelim} is a quadratic differential on the Riemann sphere $\C\PP^1$ with one pole of order $3$ and three single zeroes. In a co-ordinate $z$ centred in $0$, it therefore has the  form
\[\Psi(z)=(z-u_1)(z-u_2)(z-u_3)dz\otimes dz\]
for three distinct points $u_1,u_2,u_3$ on $\C$. Here
the pole is fixed at $\infty\in\C\PP^1$. As the triple $(u_1,u_2,u_3)$ varies in $\C^3$ we get different quadratic differentials of the same form. The condition for
the zeroes to be distinct can be reformulated as
$\prod_{i<j}(u_i-u_j)\neq 0$. Of course, the result
is independent on permutations of $u_1,u_2,u_3$ so that the
parameter space will be quotiented by the symmetric group $\Sigma_3$. Last we can translate the triple and assume that the centre of mass of these points
is the origin, i.e., $u_1+u_2+u_3=0$. The meromorphic
quadratic differential $\Psi$ is equivalently specified by two
parameters $a=(u_1u_2+u_2u_3+u_3u_1)$ and $b=-u_1u_2u_3$:
\[\Psi_{a,b}(z)=(z^3+az+b)dz\otimes dz\]
for $4a^3+27b^2\neq 0$. See \cite{ikeda} and
\cite[Section 12.1]{BS15} for a precise description of
the relevant space of quadratic differentials. Theorem
\ref{thm:BS15_iso} becomes the following
statement (Theorem \ref{BSforA2}).

\begin{theorem}\label{BSforA2}The connected component $\stab^\circ(\calD_3(A_2))$ is isomorphic to the universal cover
    $\widetilde{\cl M_3}$ of the configuration space
    \[\cl M_3:=\{(a,b)\in\C^2\,|\, 4a^3+27b^2\neq 0\}\,,\]
    i.e.,
    \[\stab^\circ(\calD_3(A_2))\simeq \widetilde{\cl M_3} \simeq \FQuad^\circ\left(\surf_\Tri\right)\,.\]
\end{theorem}

The isomorphism,
specialised here to $n=2$, $N=3$ from the paper \cite{ikeda} by Ikeda,
is first constructed from the space $\cl M_3$
to the quotient~$\stab^\circ(\calD_3(A_2))/\operatorname{sph}(\calD_3(A_2))$, and then lifted using that $\pi_1(\cl M_3)\simeq\operatorname{sph}(\calD_3(A_2))$ both coincide with a braid group. In fact, there are several ways of computing
$\stab(\calD_3(A_2))$ (see \cite{toms, kleinian}
\cite{ikeda,BS15} for details), which also apply to other Calabi--Yau dimensions and other quivers, e.g., \cite{caitlin}.
A fundamental domain of $\stab^\circ(\calD_3)$ with respect
to the action of the spherical twist group consists of
stability conditions supported on the finite hearts appearing
in \eqref{pentagon_hearts}. The projection of the forgetful map from this fundamental domain to  $\R^2$ coordinatised by
the purely imaginary part of the central charge of $S_1$ and $S_2$ is given in
Figure \ref{R2pentagon}.
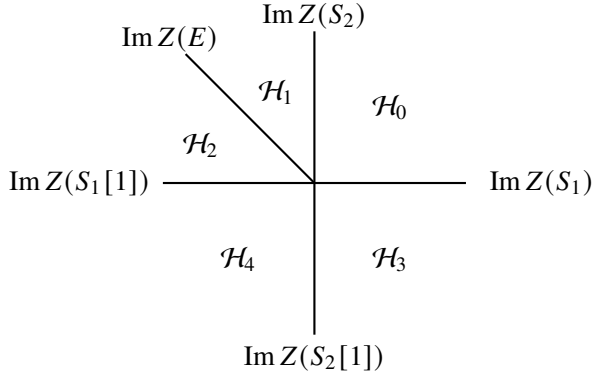
\begin{figure}[ht]
    \centering{
        \begin{tikzpicture}
            \draw[thick,] (-2,0) -- (2,0);
            \draw[thick,] (0,-2) -- (0,2);
            \draw[thick,] (0,0) -- (-1.7,1.7);
            \draw(1,1) node{$\cl H_0$};
            \draw(1,-1) node{$\cl H_3$};
            \draw(-1,-1) node{$\cl H_4$};
            \draw(-0.5,1.2) node{$\cl H_1$};
            \draw(-1.5,.5) node{$\cl H_2$};
            \draw(0,2.2) node{$\Im Z(S_2)$};
            \draw(3,0) node{$\Im Z(S_1)$};
            \draw(-1.9,1.9) node{$\Im Z(E)$};
            \draw(0,-2.3) node{$\Im Z(S_2[1])$};
            \draw(-3.1,0) node{$\Im Z(S_1[1])$};			
    \end{tikzpicture}}
    \caption{The projection of the forgetful map from $\stab(\calD_3(A_2))/\operatorname{sph}(\calD_3(A_2))$ on the $\R^2$ plane with coordinates
        the purely imaginary part of the central charge of $S_1$ and $S_2$. It represents (the projection of) five chambers of $\Stab(\calD_3(A_2))/\operatorname{sph}(\calD_3(A_2))$ and of their walls.}
    \label{R2pentagon}
\end{figure}

A straight corollary of Theorem \ref{BSforA2}
is that the connected component~
$\stab^\circ(\calD_3(A_2))$, as a topological space, is
contractible, \cite[Theorem 7.13]{ikeda}.

\subsection{Generalisations}

Some generalisations of the original Bridgeland--Smith
correspondence \eqref{kgamma} exist in the literature. They concern Ginzburg
categories of Calabi--Yau dimension greater than $3$, categories from non-simply weighted
decorated marked surfaces, and Fukaya
categories from flat surfaces.

The first is due to Ikeda \cite{ikeda} for the triangulated
categories $\pvd\Gamma_N(A_n)$ for $N\geq 3$. Similarly to
the original Bridgeland--Smith result, it is based on a correspondence
between hearts of bounded t-structures up to the action of
the $N$-spherical twist group and un-decorated $N$-angulations of a
polygon with $(N-2)(n+1)+2$ edges. Simple tilts of hearts
correspond to un-decorated flips of edges of the
$N$-angulation and to cluster mutations in the coloured
exchange graph.
Thanks to the relation between exchange graphs and coloured
exchange graphs in cluster category theory, this approach
seems to be adaptable to other Ginzburg algebras
$\Gamma_N(Q,0)$, $N\geq 3$, such that the corresponding
$N$-cluster category admits a geometric description in
terms of $N$-angulations. I am not aware of further work in
this direction.

In a similar framework, in \cite{BMQS} the hypothesis of simple weights is relaxed and an unpunctured $\surf_\w$ is associated with an
appropriate Verdier localisation $\calD(\surf_\w)$ of a Ginzburg category. A union of connected
components of $\stab(\calD(\surf_\w))$ is described in terms
of quadratic differentials with vanishing order vector $\w\geq \mathbf{1}$.

Haiden and collaborators
\cite{haiden1, haiden2} have considered quadratic differentials with exponential-type
singularities and with only simple poles and zeroes. In the
latter case a theory of counting finite-length geodesics is
initiated from a Donaldson--Thomas theory enumerating
semistable objects.
The involved categories are Fukaya
categories of surfaces with boundaries, whose objects
correspond to
a suitable collection of arcs. They are relevant in the
context of mirror symmetry, and do not come from
quivers with potential.
\medskip
\par
These constructions provide additional
examples of the relation
between central charges in the theory of stability
condition and period maps, and show that this is  not
limited to the $\CY_3$ setup.

%------
% Insert acknowledgments and information
% regarding funding at the end of the last
% section, i.e., right before the bibliography.
%------

\begin{ack}
The author thanks first and foremost the scientific committee and the organisers of ICRA 2022 for the invitation to deliver a mini-course on the subject of this article and for providing a friendly and inspiring atmosphere. Thanks are also due to Lidia Angeleri, Rosie Laking, Sergio Pavon, Matthew Pressland, Hipolito Treffinger, Jorge Vitoria for related discussions at ICRA and beyond.
\end{ack}

\begin{funding}
This work was partially supported by the
project \emph{REDCOM: Reducing complexity in algebra, logic, combinatorics}, financed by the programme \emph{Ricerca Scientifica
di Eccellenza 2018} of the Fondazione Cariverona.
\end{funding}

%------
% Insert the bibliography.
%------

\end{document}